\documentclass[reqno,letter, 11pt]{article}
\usepackage{color}
\usepackage[latin1]{inputenc}
\usepackage[english]{babel}
\usepackage{amsmath,amssymb,amsfonts,amsthm,enumerate}
\usepackage{mathrsfs}
\usepackage[T1]{fontenc}
\usepackage{epsf, epsfig}
\usepackage{graphicx}

\definecolor{refkeybis}{gray}{.65}
\definecolor{labelkeybis}{gray}{.65}
{\makeatletter
\def\SK@refcolor{\color{refkeybis}}%
\def\SK@labelcolor{\color{labelkeybis}}}

\numberwithin{equation}{section} 

\newtheorem{theorem}{Theorem}[section]
\newtheorem{lemma}[theorem]{Lemma}
\newtheorem{definition}[theorem]{Definition}
\newtheorem{remark}[theorem]{Remark}
\newtheorem{proposition}[theorem]{Proposition}
\newtheorem{corollary}[theorem]{Corollary}

\newcommand{\N}{\mathbf{N}}
\newcommand{\R}{\mathbf{R}}

\newcommand{\Haus}[1]{{\mathscr H}^{#1}} 
\newcommand{\Leb}[1]{{\mathscr L}^{#1}} 

\newcommand{\x}{\times}

\renewcommand{\>}{\rangle}
\renewcommand{\a}{\alpha}

\newcommand{\cl}{\overline}
\renewcommand{\d}{\delta}

\newcommand{\e}{\varepsilon}
\newcommand{\g}{\gamma}

\renewcommand{\l}{\lambda}

\renewcommand{\i}{\infty}
\newcommand{\p}{\partial}

\newcommand{\supp}{\operatorname{spt}}
\newcommand{\intr}{\operatorname{int}}
\newcommand{\dist}{\operatorname{dist}}
\newcommand{\diam}{\operatorname{diam}}
\newcommand{\co}{\operatorname{co}}

\newcommand\Bzero{{\bf (B0)}}
\newcommand\Bone{{\bf (B1)}}
\newcommand\Btwo{{\bf (B2)}}
\newcommand\Btwos{{\bf (B2u)}}
\newcommand\Bthree{{\bf (B3)}}

\newcommand\Athreew{{\bf (A3w)}}
\newcommand\cross{\hbox{\rm cross}}
\newcommand\dom{\hbox{\rm dom}\thinspace}
\newcommand\zero{{\mathbf 0}}

\newcommand\U{U}
\newcommand\V{V}
\newcommand\uu{u}
\newcommand\tu{{\tilde u}}
\newcommand\vv{v}
\newcommand\bo{{b^0}}
\newcommand\tb{{\tilde b}}
\newcommand\cs{{c^*}}
\newcommand\tc{{\tilde c}}

\newcommand\bq{{\bar q}}
\newcommand\tq{{\tilde q}}
\newcommand\qo{q^0}
\newcommand\bs{{\bar s}}
\newcommand\bx{{\bar x}}
\newcommand\xo{{x_0}}
\newcommand\by{{\bar y}}
\newcommand\tx{{\tilde x}}
\newcommand\ty{{\tilde y}}

\newcommand\normal{{\hat n}}
\newcommand\DASM{{\bf DASM}}

\title{
Continuity and injectivity of optimal maps\\ for non-negatively
cross-curved
costs\thanks{The authors are 
grateful to
the Institute for Pure and Applied Mathematics at UCLA,
the Institut Fourier at Grenoble, and the Fields Institute in Toronto,
for their generous hospitality during various stages of this work.
This research was supported in part by
NSERC grants 217006-03 and -08 and 
NSF grant DMS-0354729.
Y-H.K. is supported partly by 
NSF grant DMS-0635607
through the membership at Institute for Advanced Study at Princeton NJ, and also in part by NSERC grant 371642-09. Any opinions, findings
and conclusions or recommendations expressed in this material are those of authors and do not
 reflect the views of either the Natural Sciences and Engineering
Research Council of Canada (NSERC) or the United States National Science Foundation (NSF).
\copyright 2009 by the authors.
}}
\date{\today}
\author{Alessio Figalli\thanks{Department of Mathematics, University of Texas at Austin, Austin Texas USA {\tt figalli@math.utexas.edu}},
Young-Heon Kim\thanks{ Department of Mathematics, University of
British Columbia, Vancouver BC Canada {\tt yhkim@math.ubc.ca} and
School of Mathematics, Institute for Advanced Study at Princeton
NJ USA} \ and Robert J. McCann\thanks{Department of Mathematics,
University of Toronto, Toronto Ontario Canada M5S 2E4 {\tt
mccann@math.toronto.edu}}}

\begin{document}

\maketitle

\begin{abstract}
Consider transportation of one distribution of mass onto another,
chosen to optimize the total expected cost, where cost per unit
mass transported from $x$ to $y$ is given by a smooth function
$c(x,y)$.  If the source density $f^+(x)$ is bounded away
from zero and infinity in an open region $\U' \subset \R^n$,  and the target
density $f^-(y)$ is bounded away from zero and infinity on its
support $\cl \V \subset \R^n$, which is strongly $c$-convex with respect to
$\U'$, and the transportation cost $c$ is non-negatively
cross-curved, we deduce 
continuity and injectivity of the optimal map inside $\U'$ (so
that the associated potential $u$ belongs to $C^1(\U')$). This
result provides a crucial step in the low/interior regularity setting: in a
subsequent paper~\cite{FigalliKimMcCann09p}, we use it to
establish regularity of optimal maps with respect to the
Riemannian distance squared on arbitrary products of spheres. The
present paper also provides an argument required by Figalli and
Loeper to conclude in two dimensions continuity of optimal maps
under the weaker (in fact, necessary) hypothesis \Athreew\
\cite{FigalliLoeper08p}.  In higher dimensions, if the densities
$f^\pm$ are H\"older continuous,  our result permits
continuous differentiability of the map inside $U'$ (in fact, $C^{2,\alpha}_{loc}$ regularity of the associated
potential) to be deduced from the work of Liu, Trudinger and
Wang~\cite{LiuTrudingerWang09p}.
\end{abstract}

\tableofcontents

\section{Introduction}

Given probability densities $0 \le f^\pm \in L^1(\R^n)$ with respect to Lebesgue measure
$\Leb{n}$ on $\R^n$,
and a cost function $c:\R^n \times \R^n \longmapsto [0,+\infty]$,
Monge's transportation problem is to find a map $G:\R^n \longmapsto \R^n$
pushing $d\mu^+ = f^+ d\Leb{n}$ forward to $d\mu^-= f^-d\Leb{n}$
which minimizes the expected transportation cost \cite{Monge81}
\begin{equation}\label{Monge}
\inf_{G_\#\mu^+ = \mu^-} \int_{\R^n} c(x,G(x)) d\mu(x),
\end{equation}
where $G_\# \mu^+ = \mu^-$ means $\mu^-[Y] = \mu^+[G^{-1}(Y)]$ for
each Borel $Y \subset \R^n$.

In this context it is interesting to know when a map attaining this infimum exists;
sufficient conditions for this were found by
Gangbo \cite{Gangbo95} and by Levin \cite{Levin99}, extending work
of a number of authors described in \cite{GangboMcCann96} \cite{Villani09}.
One may also ask when $G$ will be smooth, in which case it must satisfy the prescribed
Jacobian equation $|\det DG(x)|=f^+(x)/f^-(G(x))$,
which turns out to reduce to a degenerate elliptic
partial differential equation of Monge-Amp\`ere type 
for a scalar potential $\uu$ satisfying $D\uu(\tx)=-D_x c(\tx,G(\tx))$.
Sufficient conditions for this were discovered by
Ma, Trudinger and Wang \cite{MaTrudingerWang05} and
Trudinger and Wang \cite{TrudingerWang07p} \cite{TrudingerWang08p},
after results for the special case $c(x,y)= |x-y|^2/2$ had been
worked out by Brenier \cite{Brenier91}, Delan\"oe \cite{Delanoe91},
Caffarelli \cite{Caffarelli90p} \cite{Caffarelli90} \cite{Caffarelli91} \cite{Caffarelli92} 
\cite{Caffarelli96b}, and Urbas \cite{Urbas97},  and for the
cost $c(x,y)=-\log |x-y|$ and measures supported on the unit sphere by
Wang~\cite{Wang96}. 

If the ratio $f^+(x)/f^-(y)$ --- although bounded away from zero
and infinity --- is not continuous, the map $G$ will not generally
be differentiable, though one may still hope for it to be
continuous. This question is not merely of academic interest,
since discontinuities in $f^\pm$ arise unavoidably in applications
such as partial transport problems \cite{CaffarelliMcCann99}
\cite{BarrettPrigozhin09} \cite{FigalliARMA} \cite{FigalliNote}.
Such results were established for the classical cost
$c(x,y)=|x-y|^2/2$ by Caffarelli \cite{Caffarelli90}
\cite{Caffarelli91} \cite{Caffarelli92}, for its restriction to
the product of the boundaries of two strongly convex sets by
Gangbo and McCann \cite{GangboMcCann00}, and for more general
costs satisfying the strong regularity hypothesis \textbf{(A3)} of
Ma, Trudinger and Wang \cite{MaTrudingerWang05}
--- which excludes the cost $c(x,y)=|x-y|^2/2$ --- by Loeper
\cite{Loeper07p}; see also \cite{KimMcCannAppendices} \cite{Liu09}
\cite{TrudingerWang08p}. Under the weaker and degenerate
hypothesis \textbf{(A3w)} of Trudinger and Wang
\cite{TrudingerWang07p}, which includes the cost
$c(x,y)=|x-y|^2/2$ (and whose necessity for regularity was shown
by Loeper \cite{Loeper07p}), such a result remains absent from the
literature; we aim to provide it below under a slight
strengthening of their condition (still including the quadratic
cost) which appeared in Kim and McCann
\cite{KimMcCann07p}\cite{KimMcCann08p}, called {\em non-negative
cross-curvature}. (Related but different families of
strengthenings were investigated by Loeper and Villani
\cite{LoeperVillani08p} and Figalli and Rifford
\cite{FigalliRifford08p}.) {Our main result is stated in
Theorem~\ref{T:Hoelder}. A number of interesting cost functions do
satisfy non-negative cross-curvature hypothesis, and have
applications in economics \cite{FigalliKimMcCann-econ09p} and
statistics \cite{Sei09p}. Examples include the Euclidean distance
between two convex graphs over two sufficiently convex sets in
$\R^n$ \cite{MaTrudingerWang05}, the Riemannian distance
squared on multiple products of round spheres (and
their Riemannian submersion quotients, including products of
complex projective spaces $\mathbf{CP}^{n}$) \cite{KimMcCann08p},
and the simple harmonic oscillator action \cite{LeeMcCann}.
In a sequel,  we apply the techniques developed here to deduce
regularity of optimal maps in the latter setting
\cite{FigalliKimMcCann09p}. Moreover, Theorem~\ref{T:Hoelder}
allows one to apply the higher interior regularity results established
by Liu, Trudinger and Wang \cite{LiuTrudingerWang09p},
ensuring in particular that the transport map is $C^\infty$-smooth
 if $f^+$ and $f^-$ are.
\\}

Most of the regularity results quoted above derive from one of two
approaches. The continuity method, used by Delano\"e, Urbas, Ma,
Trudinger and Wang,  is a time-honored technique for solving
nonlinear equations.  Here one perturbs a manifestly soluble
problem (such as $|\det DG_0(x)|=f^+(x)/f_0(G_0(x))$ with $f_0=f^+$, so that $G_0(x)=x$) to the problem of interest
($|\det DG_1(x)|=f^+(x)/f_1(G_1(x))$, $f_1=f^-$)
along a family $\{f_t\}_t$ designed to ensure the set of $t
\in[0,1]$ for which it is soluble is both open and closed.
Openness follows from linearization and non-degenerate ellipticity
using an implicit function theorem. For the non-degenerate ellipticity and closedness,
it is required to establish estimates on the size of derivatives of the solutions
(assuming such solutions exist) which depend only on information
known a priori about the data $(c,f_t)$.  In this way one obtains
smoothness of the solution $y=G_1(x)$ from the same argument which
shows $G_1$ to exist.

The alternative approach relies on first knowing existence and
uniqueness of a Borel map which solves the problem in great
generality,  and then deducing continuity or smoothness by close
examination of this map after imposing additional conditions on
the data $(c,f^\pm)$. Although precursors can be traced back to
Alexandrov \cite{Aleksandrov42b}, in the present context this
method was largely developed and refined by Caffarelli
\cite{Caffarelli90} \cite{Caffarelli91} \cite{Caffarelli92}, who
used convexity of $\uu$ crucially to localize the map $G(x) =
D\uu(x)$ and renormalize its behaviour near a point $(\tx,G(\tx))$
of interest in the borderline case $c(x,y) = - \<x, y\>$. For
non-borderline \textbf{(A3)} costs, simpler estimates suffice to
deduce continuity of $G$, as in \cite{GangboMcCann00}
\cite{CafGutHua} \cite{Loeper07p} \cite{TrudingerWang08p}; in this
case Loeper was actually able to deduce an explicit bound $\alpha
= (4n-1)^{-1}$ on the H\"older exponent of $G$ when $n>1$, which
was recently improved to its sharp value $\alpha = (2n-1)^{-1}$ by
Liu \cite{Liu09} using a technique related to the one we develop
below and discovered independently from us; both Loeper and Liu
also obtained explicit exponents $\alpha=\alpha(n,p)$ for
 $f^+ \in L^p$ with $p>n$ \cite{Loeper07p} or $p> (n+1)/2$ \cite{Liu09} and $1/f^- \in L^\infty$.
Explicit bounds on the exponent are much worse in the classical
case $c(x,y) = - \<x, y\>$ \cite{ForzaniMaldonado04},
when such exponents do not even exist unless $\log \frac{f^+(x)}{f^-(y)} \in L^\infty$ \cite{Caffarelli92} \cite{Wangcounterex}.\\

Below we extend the approach of Caffarelli to non-negatively cross-curved costs,
a class which includes the classical quadratic cost. Our idea is to add a null Lagrangian term
to the cost and exploit diffeomorphism (i.e.\ gauge) invariance
to choose coordinates which depend on the point of interest that
restore convexity of $\uu(x)$; our strengthened hypothesis then
permits us to exploit Caffarelli's approach more systematically
than Liu was able to do \cite{Liu09}. However, we still need to
overcome serious difficulties, such as getting an Alexandrov
estimate for $c$-subdifferentials (see Section~\ref{S:Alexandrov})
and dealing with the fact that the domain of the cost function
(where it is smooth and satisfies appropriate cross-curvature
conditions) may not be the whole of $\R^n$. (This situation
arises, for example, when optimal transportation occurs between
domains in Riemannian manifolds for the distance squared cost or
similar type.) The latter is accomplished using
Theorem~\ref{thm:bdry-inter}, where it is first established
that optimal transport does not send interior points to boundary
points, and vice versa, under the strong $c$-convexity hypothesis
\Btwos\ described in the next section. ( For this result to hold,
the cost needs not to satisfy the condition {\bf (A3w)}.)
Without our strengthening of Trudinger and Wang's hypothesis
\cite{TrudingerWang07p} (i.e.\ with only {{\bf (A3w)}}),  we
obtain the convexity of all level sets of $\uu(x)$ in our chosen
coordinates as Liu also did; this yields some hope of applying
Caffarelli's method and the full body of techniques systematized
in Gutierrez \cite{Gutierrez01}, but we have not been successful
at overcoming the remaining difficulties in such generality.
In two dimensions however, there is an alternate
approach to establishing continuity of optimal maps
which applies to this more general case;  it was carried out
by Figalli and
Loeper \cite{FigalliLoeper08p}, but relies on Theorem~\ref{thm:bdry-inter}, first proved below. 

\section{Main result}

Let us begin by formulating the relevant hypothesis on the cost
function $c(x,y)$ in a slightly different format than Ma,
Trudinger and Wang \cite{MaTrudingerWang05}. For each $(\tx,\ty)
\in \cl \U \times \cl \V$
assume:\\
\Bzero\  $\U \subset \R^n$ and $\V \subset \R^n$ are open and bounded and
$c \in C^4\big(\cl \U \times \cl \V\big)$; \\
\Bone\   (bi-twist)
$\left.\begin{array}{c}
        x \in \cl \U \longmapsto -D_y c(x,\ty) 
\cr     y \in \cl \V \longmapsto -D_x c(\tx,y) 
        \end{array}\right\}$ are diffeomorphisms onto their ranges; \\
\Btwo\  (bi-convex)
$\left.\begin{array}{c}
        \U_{\ty} := -D_y c(\U, \ty) \cr
        \V_{\tx} := -D_x c(\tx, \V)
        \end{array}\right\}$ are convex subsets of $\R^n$; \\
\Bthree\  (non-negative cross-curvature)
\begin{equation}\label{MTW}
\cross_{(x(0),y(0))} [x'(0),y'(0)] :=
-\frac{\partial^4}{\partial s^2 \partial t^2}\bigg|_{(s,t)=(0,0)} c(x(s),y(t)) \ge 0
\end{equation}
for every curve
$t \in[-1,1] \longmapsto \big(D_y c(x(t),y(0)),D_x c(x(0),y(t))\big) \in \R^{2n}$
which is an affinely parameterized line segment.\\

If the convex domains $\U_\ty$ and $\V_\tx$ in \Btwo\ are all
strongly convex, we say \Btwos\ holds. Here a convex set $\U
\subset \R^n$ is said to be {\em strongly} convex if there exists
a radius $R<+\infty$ (depending only on $U$,) such that each
boundary point $\tx \in
\partial \U$ can be touched from outside by a sphere of radius $R$
enclosing $\U$; i.e. $\U \subset B_R(\tx - R \normal_{\U}(\tx))$
where $\normal_\U(\tx)$ is an outer unit normal to a hyperplane
supporting $\U$ at $\tx$.  When $\U$ is smooth, this means all
principal curvatures of its boundary are bounded below by $1/R$.
Hereafter $\cl \U$ denotes the closure of $\U$, $\intr U$ denotes
its interior, $\diam U$ its diameter, and for any measure $\mu^+
\ge 0$ on $\cl \U$, we use the term {\em support} and the notation
$\supp \mu^+ \subset \cl\U$ to refer to the smallest closed set
carrying the full mass of $\mu^+$.

Condition \Bthree\  is the above-mentioned strengthening of Trudinger and Wang's
criterion {\bf (A3w)} guaranteeing smoothness
of optimal maps in the Monge transportation problem (\ref{Monge});  unlike us,
they require (\ref{MTW}) only if, in addition \cite{TrudingerWang07p},
\begin{equation}\label{nullity}
\frac{\partial^2}{\partial s \partial t}\bigg|_{(s,t)=(0,0)} c(x(s),y(t)) = 0.
\end{equation}
Necessity of Trudinger and Wang's condition for continuity was shown by
Loeper \cite{Loeper07p}, who
noted its covariance (as did \cite{KimMcCann07p} \cite{Trudinger06})
and some relations to curvature.  Their condition relaxes the hypothesis {\bf (A3)}
proposed earlier with Ma \cite{MaTrudingerWang05},
which required strict positivity of (\ref{MTW}) when (\ref{nullity}) holds.
The strengthening considered here was first studied
in a different but equivalent form by Kim and McCann \cite{KimMcCann07p},
where both the original and the modified
conditions were shown to
correspond to pseudo-Riemannian sectional curvature conditions induced by the cost $c$
on $\U \times \V$,  highlighting their invariance under reparametrization
of either $\U$ or $\V$ by diffeomorphism; see \cite[Lemma 4.5]{KimMcCann07p}.
The convexity of $\U_\ty$ required in \Btwo\  is called {\em
$c$-convexity of $U$ with respect to $\ty$} by Ma, Trudinger and
Wang (or strong $c$-convexity if \Btwos\ holds); they call curves
$x(s) \in \U$, for which $s \in [0,1] \longmapsto -D_y
c(x(s),\ty)$ is a line segment, {\em $c$-segments with respect to
$\ty$}. Similarly, $\V$ is said to be strongly $\cs$-convex with
respect to $\tx$
--- or with respect to $\cl\U$ when it holds for all $\tx \in \cl\U$ ---
and the curve $y(t)$ from \eqref{MTW} is said to be a
$\cs$-segment with respect to $\tx$.
Such curves correspond to geodesics $(x(t),\ty)$ and $(\tx, y(t))$
in the geometry of Kim and McCann. Here and throughout,  {\em line segments}
are always presumed to be affinely parameterized.

We are now in a position to summarize our main result:

\begin{theorem}[Interior continuity and injectivity of optimal maps]
\label{T:Hoelder}
Let $c \in C^4 \big(\cl \U \times \cl \V\big)$ satisfy \Bzero--\Bthree\ and \Btwos.
Fix probability densities $f^+ \in L^1\big(\U\big)$ and $f^- \in L^1\big(\V\big)$
with $(f^+/f^-) \in L^\infty\big(\U \x \V\big)$ and set $d\mu^\pm
:=f^\pm d\Leb{n}$. If the ratio $(f^-/f^+) \in
L^\infty(\U' \x \V)$ for some open set $\U' \subset \U$, then
the minimum (\ref{Monge}) is attained by a
map $G:\cl \U \longmapsto \cl \V$ whose restriction to $\U'$ is
continuous and one-to-one.
\end{theorem}

\begin{proof}
As recalled below in Section \ref{S:background}
(or see e.g.\ \cite{Villani09}) 
it is well-known by Kantorovich duality that the optimal joint
measure $\gamma \in \Gamma(\mu^+,\mu^-)$ from \eqref{Kantorovich}
vanishes outside the $c$-subdifferential \eqref{c-subdifferential}
of a potential $u=u^{c^* c}$ satisfying the $c$-convexity
hypothesis \eqref{c-transform}, and that the map $G:\cl\U
\longmapsto \cl\V$ which we seek is uniquely recovered from this
potential using the diffeomorphism \Bone\ to solve \eqref{implicit
G}. Thus the continuity claimed in Theorem \ref{T:Hoelder} is
equivalent to $\uu \in C^{1}(\U')$.

{Since $\mu^\pm$ do not charge the boundaries of $\U$ (or of
$\V$),} Lemma \ref{L:cMA properties}(e) shows the
$c$-Monge-Amp\`ere measure defined in \eqref{c-Monge-Ampere
measure} has density satisfying $|\p^c u| \le
\|f^+/f^-\|_{L^\infty(\U \times \V)}$ on $\cl \U$ and
$\|f^-/f^+\|^{-1}_{L^\infty(\U' \times \V)} \le |\p^c
u| \le \|f^+/f^-\|_{L^\infty(\U' \times \V)}$ on $\U'$. Thus $u
\in C^1(\U')$ according to Theorem \ref{T:continuity}. Injectivity
of $G$ follows from Theorem \ref{T:strict convex}, and the fact
that the graph of $G$ is contained in the set $\p^c u \subset \cl
\U \times \cl \V$ of \eqref{c-subdifferential}.
\end{proof}
{ Note that in case $f^+ \in C_c(U)$ is continuous and compactly supported,
 choosing $U' = U'_\e= \{ f^+ > \e \}$ for all $\e > 0$,  yields
 continuity and injectivity of the optimal map $y = G(x)$ throughout $U'_0$.}

Theorem \ref{T:Hoelder}
provides a necessary prerequisite for the higher interior
regularity results established by Liu, Trudinger and Wang in
\cite{LiuTrudingerWang09p} --- a prerequisite which one would
prefer to have under the weaker hypotheses \Bzero--\Btwo\ and
\Athreew. Note that these interior regularity results can be
applied to manifolds, after getting suitable stay-away-from-the-cut-locus
results: this is accomplished for multiple products of round spheres in
\cite{FigalliKimMcCann09p}, to yield the first regularity result that we
know for optimal maps on Riemannian manifolds which are not flat, yet
have some vanishing sectional curvatures.

\section{Background, notation, and preliminaries}
\label{S:background}

Kantorovich discerned \cite{Kantorovich42} \cite{Kantorovich48} that  Monge's problem
(\ref{Monge}) could be attacked by studying the linear programming problem
\begin{equation}\label{Kantorovich}
\min_{\gamma \in \Gamma(\mu^+,\mu^-)} \int_{\cl \U \times \cl \V}
c(x,y)\, d\gamma(x,y).
\end{equation}
Here $\Gamma(\mu^+,\mu^-)$ consists of the joint probability measures on
$\cl \U \times \cl \V \subset \R^n \times \R^n$ having $\mu^\pm$ for
marginals.  According to the duality theorem from linear programming,
the optimizing measures $\gamma$ vanish outside the zero set of
$\uu(x) + \vv(y) + c(x,y) \ge 0$ for some pair of functions
$(\uu,\vv) = (\vv^{c},\uu^\cs)$ satisfying
\begin{equation}\label{c-transform}
\vv^{c}(x) := \sup_{y \in \cl \V} -c(x,y) - \vv(y), \qquad
\uu^\cs(y) := \sup_{x \in \cl \U} -c(x,y) - \uu(x);
\end{equation}
these arise as optimizers of the dual program. This zero set is
called the $c$-subdifferential of $\uu$, and denoted by
\begin{equation}\label{c-subdifferential}
\partial^c \uu = \{(x,y) \in \cl \U \times \cl \V \mid \uu(x) + \uu^\cs(y) + c(x,y) = 0\};
\end{equation}
we also write $\partial^c \uu(x) := \{y \mid (x,y) \in \partial^c
\uu\}$, and $\partial^{\cs}\uu^\cs(y) := \{ x \mid (x,y) \in
\partial^c \uu\}$, and $\partial^c \uu(X) := \cup_{x \in
X} \partial^c \uu(x)$ for $X \subset \R^n$. Formula
(\ref{c-transform}) defines a generalized Legendre-Fenchel
transform called the {\em $c$-transform};  any function satisfying
$\uu = \uu^{\cs c} :=(\uu^\cs)^{c}$ is said to be {\em
$c$-convex}, which reduces to ordinary convexity in the case of
the cost $c(x,y) = - \<x, y\>$. In that case $\partial^c \uu$
reduces to the ordinary subdifferential $\partial \uu$ of the
convex function $\uu$,  but more generally we define
\begin{equation}\label{subdifferential}
\partial \uu
:= \{(x,p) \in \cl \U \times \R^n \mid \uu(\tx) \ge \uu(x) +
\langle p, \tx-x\rangle + o(|\tx-x|) {\rm\ as\ } \tx \to x\},
\end{equation}
$\partial\uu(x) := \{p \mid (x,p) \in \partial \uu\}$, and
$\partial \uu(X) := \cup_{x \in X} \partial \uu(x)$. Assuming $c
\in C^2\big(\cl \U \times \cl \V\big)$ (which is the case if
\Bzero\ holds),  any $c$-convex function $\uu=\uu^{c^*c}$ will be
semi-convex,  meaning its Hessian admits a bound from below $D^2
\uu \ge - \|c\|_{C^2}$ in the distributional sense;  equivalently,
$\uu(x) + \|c\|_{C^2}|x|^2/2$ is convex on each ball in $\U$
\cite{GangboMcCann96}.  In particular,  $\uu$ will be
twice-differentiable $\Leb{n}$-a.e. on $\U$ in the sense of
Alexandrov.

As in \cite{Gangbo95} \cite{Levin99} \cite{MaTrudingerWang05},
hypothesis \Bone\  shows the map $G:\dom D\uu \longmapsto \cl \V$ is uniquely
defined on the set $\dom D\uu \subset \cl\U$ of differentiability for $\uu$ by
\begin{equation}\label{implicit G}
D_x c(\tx,G(\tx)) = - D\uu(\tx).
\end{equation}
The graph of $G$,  so-defined,  lies in $\partial^c \uu$.
The task at hand is to show continuity and injectivity of $G$ ---
the former being equivalent to $\uu \in C^{1}(\U)$ ---
by studying the relation $\partial^c \uu \subset \cl \U \times \cl \V$.

To this end,  we define a Borel measure $|\partial^c \uu|$ on $\R^n$ associated to
$\uu$ by
\begin{equation}\label{c-Monge-Ampere measure}
|\partial^c \uu|(X) := \Leb{n}(\partial^c \uu(X))
\end{equation}
for each $X \subset \R^n$;
it will be called the  {\em $c$-Monge-Amp\`ere measure} of $\uu$.
{(Similarly, we define $|\partial u|$.)}
We use the notation $|\partial^c \uu| \ge  \lambda$ on $\U'$ as a shorthand
to indicate $|\partial^c \uu|(X) \ge \lambda \Leb{n}(X)$ for each $X \subset \U'$;
similarly, $|\partial^c \uu| \le \Lambda$ indicates
$|\partial^c \uu|(X) \le \Lambda \Leb{n}(X)$.
As the next lemma shows, uniform bounds above and below on the marginal densities
of a probability measure $\gamma$ vanishing outside $\partial^c \uu$
imply similar bounds on $|\partial ^c \uu|$.

\begin{lemma}[Properties of $c$-Monge-Amp\`ere measures]\label{L:cMA properties}
Let $c$ satisfy \Bzero-\Bone, while $u$ and $u_k$ denote
$c$-convex functions for each $k \in \N$. Fix $\tx \in \cl X$ and
constants $\lambda,\Lambda > 0$.
\\(a) Then
$\p^c u(\cl \U)\subset \cl \V$ and $|\partial^c u|$ is a
Borel measure of total mass $\Leb{n}\big(\cl V\big)$ on $\cl U$.
\\(b) If $u_k \to u_\infty$ uniformly,  then $u_\infty$ is $c$-convex and
$|\p^c u_k| \rightharpoonup |\p^c u_\infty|$ weakly-$*$ in
the duality against continuous functions on $\cl \U \times \cl \V$.
\\(c) If $u_k(\tx)=0$ for all $k$, then the functions $u_k$ converge uniformly if and only if
the measures $|\p^c u_k|$ converge weakly-$*$. 
\\ (d) If $|\p^c u| \le \Lambda$ on $\cl \U$, then
$|\partial^\cs u^\cs| \ge 1/\Lambda$ on $\cl \V$.
\\ (e) If a probability measure $\gamma \ge 0$
vanishes outside $\p^c \uu \subset \cl\U \times \cl\V$,
and has marginal densities $f^\pm$, 
then $f^+ \ge \lambda$ on $\U' \subset \cl \U$ and
$f^- \le \Lambda$ on $\cl \V$ imply $|\p^c \uu| \ge \lambda/\Lambda$ on $\U'$,
whereas $f^+ \le \Lambda$ on $\U'$ and $f^- \ge \lambda$ on $\cl \V$ imply
$|\p^c \uu| \le \Lambda / \lambda$ on $\U'$.
\end{lemma}

\begin{proof}
(a) The fact $\p^c u(\cl \U) \subset \cl\V$ is an immediate consequence of definition
(\ref{c-subdifferential}).
Since $c \in C^1(\cl \U \times \cl \V)$,
the $c$-transform $\vv=\uu^\cs:\cl V \longmapsto \R$ defined by (\ref{c-transform})
can be extended to a Lipschitz function on a neighbourhood of $\cl \V$,  hence
Rademacher's theorem asserts $\dom D \vv$ is a set of full Lebesgue measure in $\cl \V$.
Use \Bone\ to define the unique solution $F:\dom D\vv \longmapsto \cl \U$ to
$$
D_y c(F(\ty),\ty) = - D \vv(\ty).
$$
As in \cite{Gangbo95} \cite{Levin99}, the vanishing of 
$\uu(x) + \vv(y) + c(x,y) \ge 0$ implies $\partial^{\cs} \vv(\ty) = \{F(\ty)\}$,
at least for all points $\ty \in \dom D\vv$ where $\cl \V$ has Lebesgue density
greater than one half. For Borel $X \subset \R^n$,
this shows $\partial^c \uu(X)$ differs from the Borel set $F^{-1}(X)\cap \cl\V$ by
a $\Leb{n}$ negligible subset of $\cl \V$,  whence $|\partial^c \uu| = F_{\#}\big(\Leb{n} \lfloor_{\cl \V}\bigr)$
so claim (a) of the lemma is established.

(b) Let $\|u_k - u_\infty\|_{L^\infty(\cl U)} \to 0$. It is not
hard to deduce $c$-convexity of $u_\infty$, as in e.g.\ 
\cite{FigalliKimMcCann-econ09p}. Define $v_k = u_k^\cs$ and $F_k$
on $\dom D v_k \subset \cl\V$ as above, so that $|\p^c u_k| =
F_{k\#} \big(\Leb{n} \lfloor_{\cl \V}\bigr)$.  Moreover, $v_k \to
v_{\infty}$ in $L^{\infty}(V)$, where $v_{\infty}$ is the
$c^{*}$-dual to $u_{\infty}$. The uniform semiconvexity of $v_k$
(i.e. convexity of $v_k(y) + \frac{1}{2}\|c\|_{C^2}|y|^2$) ensures
pointwise convergence of $D v_k \to Dv_\infty$ $\Leb{n}$-a.e.\ on
$\cl \V$. From $D_y c(F_k(\ty),\ty) = - D \vv_k(\ty)$ we deduce
$F_k \to F_\infty$ $\Leb{n}$-a.e.\ on $\cl \V$. This is enough to
conclude $|\p^c u_k| \rightharpoonup |\p^c u_k|$,  by testing the
convergence against continuous functions and applying Lebesgue's
dominated convergence theorem.

(c) To prove the converse,  suppose $u_k$ is a sequence of
$c$-convex functions which vanish at $\tx$ and $|\p^c u_k|
\rightharpoonup \mu_\infty$ weakly-$*$.  Since the $u_k$ have
Lipschitz constants dominated by $\|c\|_{C^1}$ and $\cl U$ is
compact, any subsequence of the $u_k$ admits a convergent further
subsequence by the Ascoli-Arzel\`a Theorem. A priori,  the limit
$u_\infty$ might depend on the subsequences, but (b) guarantees
$|\p^c u_\infty| = \mu_\infty$,  after which \cite[Proposition
4.1]{Loeper07p} identifies $u_\infty$ uniquely in terms of
$\mu^+=\mu_\infty$ and $\mu^- = \Leb{n} \lfloor_{\cl \V}$, up to
an additive constant; this arbitrary additive constant is fixed by
the condition $u_\infty(\tx)=0$.  Thus the whole sequence $u_k$
converges uniformly.

(e) Now assume a finite measure $\gamma \ge 0$ vanishes outside
$\p^c \uu$ and has marginal densities $f^\pm$.
Then the second marginal $d\mu^- := f^- d\Leb{n}$ of $\gamma$ is
absolutely continuous with respect to Lebesgue and $\gamma$
vanishes outside the graph of $F:\overline \V \longmapsto \U$,
whence $\gamma = (F \times id)_\# \mu^-$ by e.g.\ \cite[Lemma
2.1]{AhmadKimMcCann09p}. (Here $id$ denotes the identity map,
restricted to the domain $\dom D\vv$ of definition of $F$.)
Recalling that $|\partial^c \uu| = F_{\#}\big(\Leb{n} \lfloor_{\cl
\V}\bigr)$ (see the proof of (a) above), for any Borel $X \subset
\U'$ we have
$$
\lambda|\p^c u|(X)
= \lambda\Leb{n}(F^{-1}(X))
\le \int_{F^{-1}(X)} f^-(y) d\Leb n (y)
= \int_X f^+(x) d\Leb n (x)
\le \Lambda \Leb{n}(X)
$$
whenever $\lambda \le f^-$ and $f^+ \le \Lambda$.  We can also
reverse the last four inequalities and interchange $\lambda$ with
$\Lambda$ to establish claim (e) of the lemma.

(d) The last point remaining follows from (e) by taking $\gamma =
(F \times id)_\#\Leb n$. Indeed
an upper bound $\lambda$ on $|\p^c u|=F_\# \Leb n$
throughout $\cl U$ and lower bound $1$ on $\Leb{n}$ translate into
a lower bound $1/\lambda$ on $|\p^\cs u^\cs|$,  since the
reflection $\gamma^*$ defined by $\gamma^*(Y \times X) := \gamma
(X \times Y)$ for each $X \times Y \subset U \times V$ vanishes
outside $\p^\cs u^\cs$ and has second marginal absolutely
continuous with respect to Lebesgue by the hypothesis $|\partial^c
u| \le \lambda$.
\end{proof}

\begin{remark}[Monge-Amp\`ere type equation]{\rm
Differentiating (\ref{implicit G}) formally with respect to $\tx$
and recalling  $|\det DG(\tx)| = f^+(\tx) /f^-(G(\tx))$  yields
the Monge-Amp\`ere type equation
\begin{equation}\label{Monge-Ampere type equation}
\frac{\det [D^2 \uu(\tx) + D^2_{xx} c(\tx,G(\tx))]}{ |\det D^2_{xy} c(\tx,G(\tx))|}
= \frac{f^+(\tx)}{f^-(G(\tx))}
\end{equation}
on $\U$, where $G(\tx)$ is given as a function of $\tx$ and $Du(\tx)$
by (\ref{implicit G}).
Degenerate ellipticity follows from the fact
that $y = G(x)$ produces equality in $\uu(x) + \uu^\cs(y) + c(x,y) \ge 0$.
A condition under which $c$-convex weak-$*$ solutions are known to
exist 
is given by
$$
\int_{\cl\U} f^+(x) d\Leb{n}(x) = \int_{\cl \V} f^-(y) d\Leb{n}(y). \\
$$
The boundary condition $\p^c \uu(\cl U) \subset \cl V$
which then guarantees $Du$ to be uniquely determined $f^+$-a.e.\ is built into
our definition of $c$-convexity. In fact, \cite[Proposition 4.1]{Loeper07p}
shows $u$ to be uniquely determined up to additive constant if either $f^+>0$ or
$f^->0$ $\Leb{n}$-a.e. on its connected domain, $\U$ or $\V$.}
\end{remark}

A key result we shall exploit several times is a maximum principle
first deduced from Trudinger and Wang's work
\cite{TrudingerWang07p}  by Loeper; see
\cite[Theorem 3.2]{Loeper07p}. A simple and direct proof, and also an extension can be found in
\cite[Theorem 4.10]{KimMcCann07p}, where the principle was also
called `double-mountain above sliding-mountain' (\DASM).  Other
proofs and extensions appear in \cite{TrudingerWang08p} \cite{TrudingerWang08q}
\cite{Villani09} \cite{LoeperVillani08p} \cite{FigalliRifford08p}:

\begin{theorem}[Loeper's maximum principle `\DASM']\label{T:DASM}
Assume \Bzero--\Btwo\ and \Athreew\ and fix $x,\tx \in \cl \U$.
If $t\in[0,1] \longmapsto -D_x c(\tx,y(t))$ is a line segment then
$f(t) := -c(x,y(t)) + c(\tx,y(t)) \le \max\{f(0),f(1)\}$ for all $t \in [0,1]$.
\end{theorem}

It is through this theorem and the next that hypothesis \Athreew\ and the
non-negative cross-curvature hypothesis \Bthree\
enter crucially. Among the many corollaries Loeper deduced from this result,
we shall need two. Proved in \cite[Theorem 3.1 and Proposition 4.4]{Loeper07p}
(alternately \cite[Theorem 3.1]{KimMcCann07p} and \cite[A.10]{KimMcCannAppendices}),
they include the $c$-convexity of the so-called {\em contact set}
(meaning the $c^*$-subdifferential at a point), and a local to global principle.

\begin{corollary}\label{C:local-global}
Assume \Bzero--\Btwo\ and \Athreew\ and fix $(\tx,\ty) \in \cl \U
\x \cl\V$. If $u$ is $c$-convex then $\p^c u(\tx)$ is $c^*$-convex
with respect to $\tx \in \U$, i.e.\ $-D_x c(\tx, \p^c u(\tx))$
forms a convex subset of $T^*_\tx \U$. Furthermore,  any local
minimum of the map $x\in\U \longmapsto u(x) + c(x,\ty)$ is a
global minimum.
\end{corollary}

As shown in \cite[Corollary 2.11]{KimMcCann08p},
the strengthening \Bthree\ of hypothesis \Athreew\ improves
the conclusion of Loeper's maximum principle. 
This improvement asserts that the altitude  $f(t,x)$ at each point
of the evolving landscape then accelerates as a function of $t\in[0,1]$:

\begin{theorem}[{\bf Time-convex DASM}]\label{T:time-convex DASM}
Assume \Bzero--\Bthree\ and fix $x,\tx \in \cl \U$.
If $t\in[0,1] \longmapsto -D_xc(\tx,y(t))$ is a line segment then
$t \in [0,1] \longmapsto f(t) := -c(x,y(t)) + c(\tx,y(t))$ is convex.
\end{theorem}

\begin{remark}\label{rmk:symmetry}{\rm
Since all assumptions \Bzero--\Bthree\ and \Athreew\ on the cost
are symmetric in $x$ and $y$,
all the results above still hold when exchanging $x$ with $y$.}
\end{remark}

\section{Cost-exponential coordinates, null Lagrangians, and affine renormalization}\label{S:notation}
In this section, we set up the notation for the rest of the paper.
%
Recall that $c \in C^4(\cl\U \times \cl \V)$ is a non-negatively
cross-curved cost function satisfying \Bone--\Bthree\ on a pair of
bounded domains $\U$ and $\V$ which are strongly $c$-convex with
respect to each other \Btwos.

Fix $\lambda,\Lambda >0$ and an open domain $\U^\lambda \subset
\U$, and let $\uu$ be a $c$-convex solution of the
$c$-Monge-Amp\`ere equation
\begin{equation}
\label{eq:cMA}
\left\{ \begin{array}{ll}
\l \Leb{n}
\leq
|\p^c\uu| \leq \frac{1}{\l}\Leb{n} 
&\text{in } \U^\lambda \subset \U,
\\ |\p^c\uu| \le \Lambda \Leb{n} &\text{in } \cl\U.
\end{array}
\right.
\end{equation}
We sometimes abbreviate \eqref{eq:cMA} by writing $|\p^c \uu| \in
[\lambda, 1/\lambda]$. In the following sections,  we will prove
interior differentiability of $\uu$ on $\U^\lambda$, that is $\uu
\in C^{1}(\U^\lambda)$; see Theorem \ref{T:continuity}.

Throughout $D_y$ will denote the derivative with respect to the variable $y$,
and iterated subscripts as in $D^2_{xy}$ denote iterated derivatives.
We also use
\begin{eqnarray}\label{bi-Lipschitzbound}
\beta^\pm_c &=&\beta^\pm_c(U \times V) := \ \ \| (D^2_{xy} c)^{\pm 1}\|_{L^\infty(U \times V)}\\
\gamma^\pm_c &=& \gamma^\pm_c(U \times V) := \|\det (D^2_{xy} c)^{\pm 1}\|_{L^\infty(U \times V)}
\label{Jacobian bound}
\end{eqnarray}
to denote the bi-Lipschitz constants $\beta^\pm_c$ of the
coordinate changes (\ref{q of x}) and the Jacobian bounds
$\gamma^\pm_c$ for the same transformation. Notice $\gamma^+_c
\gamma^-_c \ge 1$ for any cost satisfying \Bone, and equality
holds whenever the cost function $c(x,y)$ is quadratic. So the
parameter $\gamma^+_c\gamma^-_c$ crudely quantifies the departure
from the quadratic case. The inequality $\beta^+_c \beta^-_c \ge
1$ is much more rigid,  equality implying $D^2_{xy} c(x,y)$ is the
identity matrix, and not merely constant.

\subsection{Choosing coordinates which convexify $c$-convex functions}\label{S:transform}

In the current subsection, we introduce an important transformation (mixing
dependent and independent variables)
for the cost $c(x,y)$ and potential $\uu(x)$, which plays a crucial role in the subsequent
analysis. This change of variables and its most relevant properties are encapsulated in the
following definition and theorem.
{\em In the sequel, whenever we use the expression  $\tilde c(q, \cdot)$ or $\tilde u(q)$
we refer to the modified cost function and convex potential defined here, unless otherwise stated.}
Since properties \Bzero--\Bthree, \Athreew\ and \Btwos\ were shown to be tensorial in nature
(i.e.\ coordinate independent) in \cite{KimMcCann07p} \cite{Loeper07p}, the modified cost $\tilde c$ inherits these
properties from the original cost $c$ with one exception:  (\ref{x of q})
defines a $C^3$ diffeomorphism $q \in \cl \U_{\ty}  \longmapsto x(q) \in \cl \U $, so
the cost $\tilde c \in C^3(\cl \U_{\ty} \times \cl \V)$ may not be $C^4$ smooth.
However, its definition reveals that
we may still differentiate $\tilde c$ four times as long as no more than three of the
four derivatives fall on the variable $q$,  and it leads
to the same geometrical structure (pseudo-Riemannian curvatures, including (\ref{MTW}))
as the original cost $c$ since the metric tensor and symplectic form defined
in \cite{KimMcCann07p} involve only mixed derivatives $D^2_{qy} \tilde c$,  and
therefore remain $C^2$ functions of the coordinates $(q,y) \in \cl \U_\ty \times \cl \V$.

\begin{definition}[Cost-exponential coordinates and apparent properties]
\label{D:cost exponential}
Given $c \in C^4\big(\cl\U \times \cl\V\big)$ strongly twisted \Bzero--\Bone, we refer to
the coordinates $(q,p) \in \cl \U_{\ty} \times \cl \V_{\tx}$ defined by
\begin{equation}\label{q of x}
q = q(x) = - D_y c(x,\ty), \qquad p = p(y) = -D_x c(\tx,y),
\end{equation}
as the \emph{cost exponential coordinates} from $\ty \in \cl \V$ and $\tx \in \cl\U$ respectively.
We denote the inverse diffeomorphisms by
$x:\cl \U_{\ty} \subset T^*_{\ty} \V\longmapsto \cl\U$ and
$y:\cl \V_{\tx} \subset T^*_{\tx} \U \longmapsto \cl\V$; they satisfy
\begin{equation}\label{x of q}
q = - D_y c(x(q),\ty), \qquad p = -D_x c(\tx,y(p)).
\end{equation}
The cost $\tilde c(q,y) = c(x(q),y) - c( x(q),\ty)$ is called the
\emph{modified cost at $\ty$}. A subset of $\cl \U$ or function
thereon is said to \emph{appear from $\ty$} to have
property $A$, 
if it has property $A$ when expressed in the coordinates $q \in
\cl \U_{\ty}$.
\end{definition}

\begin{remark}{\rm
Identifying the cotangent vector $0 \oplus q$
with the tangent vector $Q \oplus 0$ to $\U \times \V$ using
the pseudo-metric of Kim and McCann \cite{KimMcCann07p}
shows $x(q)$ %
to be the projection to $\U$ of the pseudo-Riemannian
exponential map $\exp_{(\tx,\ty)} Q \oplus 0$;  similarly
$y(p)$
is the projection to $\V$ of $\exp_{(\tx,\ty)} 0 \oplus P$.
Also, $x(q) =: \cs$-$\exp_{\ty}q$ and $y(p) =: c$-$\exp_{\tx} p$
in the notation of Loeper \cite{Loeper07p}.}
\end{remark}

Our first contribution is the following theorem.
For a non-negatively cross-curved cost \Bthree,  it shows that
any $\tilde c$-convex potential appears convex from $\ty \in \cl \V$.
Even if the cost function is weakly regular \Athreew,
the level sets of the $\tilde c$-convex potential appear convex from $\ty$,
as was discovered independently from us by Liu \cite{Liu09}, and
exploited by Liu with Trudinger and Wang \cite{LiuTrudingerWang09p}.
Note that although the difference between the cost $c(x,y)$ and the modified cost
$\tilde c(q,y)$ depends on $\ty$,  they differ by a null Lagrangian $c(x,\ty)$
which --- being independent of $y \in \V$ --- does not affect the question
of which maps $G$ attain the infimum (\ref{Monge}).
Having a function with convex level sets is a useful starting point,
since it enables us to apply Caffarelli's affine renormalization of convex sets
approach and a full range of techniques from Gutierrez \cite{Gutierrez01}
to address the regularity of $c$-convex potentials.

\begin{theorem}[Modified $c$-convex functions appear convex]\label{thm:apparently convex}
Let $c \in C^4 \big(\cl \U \times \cl\V\big)$ satisfying
\Bzero--\Btwo\ be weakly regular \Athreew.
If $\uu = \uu^{c^* c}$ is $c$-convex on $\cl \U$,  then
$\tu(q) = \uu(x(q)) + c(x(q),\ty)$ has convex level sets,
as a function of the cost exponential coordinates
$q \in \cl \U_{\ty}$ from $\ty \in \cl \V$.
If, in addition, $c$ is non-negatively cross-curved \Bthree\ then $\tu$ is convex on
$\cl \U_{\ty}$.
In either case $\tu$ is minimized at $q_0$ if $\ty\in \p^c \uu(x(q_0))$.
Furthermore, $\tu$ is $\tilde c$-convex with respect to the modified cost
$\tilde c(q,y) := c(x(q),y) - c(x(q),\ty)$ on
$\cl \U_{\ty} \times \cl \V$, and
$\p^{\tilde c} \tu(q) = \p^c \uu(x(q))$ for all $q \in \cl \U_\ty$.
\end{theorem}

\begin{proof}
The final sentences of the theorem are elementary:
$c$-convexity $\uu=\uu^{c^*c}$ asserts
$$
\uu(x)=\sup_{y \in \cl \V} -c(x,y) -\uu^\cs(y) \quad {\rm and} \quad
\uu^\cs(y) = \sup_{q \in \cl \U_{\ty}} -c (x(q),y) - \uu(x(q))
= \tu^{\tilde c^*}(y)
$$
from (\ref{c-transform}), hence
\begin{eqnarray*}
\tu(q) &=& \sup_{y \in \cl \V} -c(x(q),y) + c(x(q),\ty) -\uu^\cs(y)
\\      &=& \sup_{y \in \cl \V} - \tilde c (q, y) - \tu^{\tilde c^*}(y),
\end{eqnarray*}
and $\p^{\tilde c} \tu(q) = \p^c \uu(x(q))$ since all three
suprema above are attained at the same $y \in \cl \V$.
Taking $y=\ty$ reduces the inequality $\tu(q) + \tu^{\tilde c^{*}}(y) + \tilde c(q,y) \ge 0$
to $\tu(q) \ge - \tu^{\tilde c^{*}}(\ty)$ ,  with equality precisely if
$\ty \in \partial^{\tilde c} \tu(q)$.  It remains to address the convexity claims.

Since the supremum $\tu(q)$ of a family of convex functions is again convex,
it suffices to establish the convexity of
$q \in \cl\U_{\ty} \longmapsto -\tilde c(q,y)$ for each $y \in \cl \V$
under hypothesis \Bthree.  For a similar reason, it suffices to establish the level-set
convexity of the same family of functions under hypothesis \Athreew.

First assume \Athreew.  Since
\begin{equation}\label{special geodesic}
D_y \tilde c (q,\ty) = D_y c (x(q),\ty):=-q
\end{equation}
we see that $\tilde c$-segments in $\cl \U_{\ty}$ with respect to $\ty$
coincide with ordinary line segments. Let $q(s) = (1-s) q_0 + s q_1$ be any line
segment in the convex set $\cl \U_{\ty}$.
Define $f(s,y) := -\tilde c(q(s),y) = - c(x(q(s)),y) + c(x(q(s)),\ty)$.
Loeper's maximum principal (Theorem \ref{T:DASM} above, see also Remark \ref{rmk:symmetry})
asserts $f(s,y) \le \max\{f(0,y),f(1,y)\}$,
which implies convexity of each set
$\{q \in \cl \U_{\ty} \mid -\tilde c(q,y) \le const\}$.
Under hypothesis \Bthree, Theorem \ref{T:time-convex DASM} 
goes on to assert convexity of $s \in [0,1] \longmapsto f(s,y)$ as desired.
\end{proof}

The effect of this change of gauge on Jacobian inequalities is summarized in a corollary:

\begin{corollary}[Transformed $\tc$-Monge-Amp\`ere inequalities]
\label{C:Jacobian transform}
Using the hypotheses and notation of Theorem \ref{thm:apparently convex},
if $|\p^c u| \in [\lambda,\Lambda] \subset [0,\infty]$ on $\U' \subset \cl \U$,
then $|\p^\tc \tu| \in [\lambda/\gamma^+_c,\Lambda\gamma^-_c]$ on $\U'_\ty 
= -D_yc(\U',\ty)$, where $\gamma^\pm_c = \gamma^\pm_c(\U' \times
V)$ and $\beta^\pm_c = \beta^\pm_c(\U' \times V)$ are defined in
\eqref{bi-Lipschitzbound}--\eqref{Jacobian bound}. Furthermore,
$\gamma^\pm_\tc := \gamma^\pm_\tc(\U'_\ty \times \V) \le
\gamma^+_c \gamma^-_c$ and $\beta^\pm_\tc := \beta^\pm_\tc(\U'_\ty
\times \V) \le \beta^+_c \beta^-_c$.
\end{corollary}

\begin{proof}
From the Jacobian bounds $|\det D_x q(x)| \in
[1/\gamma^-_c,\gamma^+_c]$ on $\U'$, we find $\Leb{n}(X)/
\gamma^-_c \le \Leb{n}(q(X)) \le \gamma^+_c \Leb{n}(X)$ for each
$X \subset \U'$. On the other hand,  Theorem \ref{thm:apparently
convex} asserts $\p^\tc \tu(q(X)) = \p^c u(X)$,  so the claim
$|\p^\tc \tu| \in [\lambda/\gamma^+_c,\Lambda\gamma^-_c]$ follows
from the hypothesis $|\p^c u| \in [\lambda,\Lambda]$, by
definition (\ref{c-Monge-Ampere measure}) and the fact that $q:\cl
\U \longrightarrow \cl \U_\ty$ from \eqref{q of x} is a
diffeomorphism; see \Bone. The bounds $\gamma^\pm_\tc \le
\gamma^+_c \gamma^-_c$ and $\beta^\pm_\tc \le \beta^+_c \beta^-_c$
follow from $D^2_{qy} \tc (q,y) = D^2_{xy} c(x(q),y) D_q x(q)$ and
$D_q x(q) = - D^2_{xy} c(x(q),\tilde y)^{-1}$.
\end{proof}

\subsection{Affine renormalization}\label{SS:renormal}
The renormalization of a function $\tu$ by an affine
transformation $L:\R^n \to \R^n$ will be useful in Section~\ref{S:Alexandrov} to prove our Alexandrov type estimates.
Let us therefore record the following observations. Define
\begin{equation}\label{renormalized solution}
\tu^*(q)=|\det L|^{-2/n} \tu(Lq).
\end{equation}
Here $\det L$ denotes the Jacobian determinant of $L$, i.e. the determinant of the linear part of $L$.

\begin{lemma}[Affine invariance of $\tc$-Monge-Amp\`ere measure]
Assuming \Bzero--\Bone,
given a $\tc$-convex function $\tu: \U_\ty \longmapsto \R$ and affine bijection $L:\R^n \longmapsto \R^n$,
define the renormalized potential $\tu^*$ by \eqref{renormalized solution} and renormalized cost
\begin{equation}\label{renormalized cost}
\tc_*(q,y)=|\det L|^{-2/n} \tilde c(Lq,L^*y)
\end{equation}
using the adjoint $L^*$ to the linear part of $L$.  Then, for all
Borel $Z \subset \cl \U_\ty$,
\begin{align}\label{compare p var with p var *}
|\p \tu^*| (L^{-1} Z) &=  |\det L|^{-1} |\p \tu| (Z),
\\|\p^{\tilde c_*} \tu^*| (L^{-1} Z) &=  |\det L|^{-1} |\p^\tc \tu| (Z).
\label{renormalized cMA measure}
\end{align}
\end{lemma}

\begin{proof}
From \eqref{subdifferential} we see $\bar p \in \p \tu(\bq)$ if
and only if $|\det L|^{-2/n} L^* \bar p \in \p \tu^*(L^{-1} \bq)$,
thus \eqref{compare p var with p var *} follows from $\p
\tu^*(L^{-1} Z) = |\det L|^{-2/n} L^* \big(\p \tu (Z)\big)$.
Similarly,  since \eqref{c-transform} yields $(\tu^*)^{\tilde
c^*_*}(y) = |\det L|^{-2/n} \tu^{\tilde c^*}(L^*y)$, we see $\by
\in \p^\tc \tu(\bq)$ is equivalent to $|\det L|^{-2/n} L^* \bar y
\in \p^{\tilde c_*} \tu^*(L^{-1} \bq)$ from
\eqref{c-subdifferential} (and Theorem \ref{thm:apparently
convex}), whence  $\p^{\tilde c_*} \tu^*(L^{-1} Z) = |\det
L|^{-2/n} L^* \big(\p^{\tc} \tu^*(Z)\big)$ to establish
\eqref{renormalized cMA measure}.
\end{proof}

As a corollary to this lemma,  we recover the affine invariance not only of the
Monge-Amp\`ere equation satisfied by $\tu(q)$ --- but also of the $\tc$-Monge-Amp\`ere equation
it satisfies --- under coordinate changes on $\V$
(which induce linear transformations $L$ on $T^*_\ty V$ and $L^*$ on $T_\ty V$):
for $q \in \U_\ty$,
$$
\frac{d|\p \tu^*|}{d\Leb{n}} (L^{-1} q) =  \frac{d|\p \tu|}{d\Leb{n}} (q)
\quad{\rm and}\quad
\frac{d|\p^{\tc_*} \tu^*|}{d\Leb{n}} (L^{-1} q) =  \frac{d|\p^\tc \tu|}{d\Leb{n}} (q).
$$

\section{Strongly $c$-convex interiors and boundaries not mixed by $\p^c \uu$}
\label{S:mapping}

The subsequent sections of this paper are largely devoted to ruling out exposed points
in $\U_\ty$ of sets on which ordinary convexity of the $\tc$-convex potential from
Theorem \ref{thm:apparently convex} fails to be strict.
This current section rules out exposed points on the boundary of $\U_\ty$.
We do this by proving an important topological property of the (multi-valued) mapping
$\p^c \uu \subset \cl \U \times \cl \V$.
Namely, we show that the subdifferential $\p^c \uu$ maps interior points of
$\supp |\p^c u| \subset \cl U$ 
only to interior 
points of $\V$, under hypothesis \eqref{eq:cMA}, and conversely that $\p^c \uu$
maps boundary points of $\U$ only to boundary points of $\V$.
This theorem may be of independent interest,
and was required by Figalli and Loeper to conclude their continuity
result concerning maps of the plane which optimize \Athreew\ costs \cite{FigalliLoeper08p}.

This section does not use the cross-curvature condition \Bthree\
(nor {\bf A3w}) on the cost function $c \in C^4(\cl \U \times \cl
\V)$, but relies crucially on the \emph{strong} $c$-convexity
\Btwos\ of its domains $U$ and  $V$ (but importantly, not
on $\supp |\partial^{c} u |$).  No analog for Theorem
\ref{thm:bdry-inter} was needed by Caffarelli to establish
$C^{1,\alpha}$ regularity of convex potentials $\uu(x)$ whose
gradients optimize the classical cost $c(x,y) = - \<x,y\>$
\cite{Caffarelli92}, since in that case he was able to take
advantage of the fact that the cost function is smooth on the
whole of $\R^n$ to chase potentially singular behaviour to
infinity.  (One general approach to showing regularity of
solutions for degenerate elliptic partial differential equations
is to exploit the threshold-hyperbolic nature of the solution to
try to follow either its singularities or its degeneracies to the
boundary,  where they can hopefully be shown to be in
contradiction with boundary conditions;  the {\em degenerate}
nature of the ellipticity precludes the possibility of {\em purely
local} regularizing effects.)

\begin{theorem}[Strongly $c$-convex interiors and boundaries not mixed by $\p^c \uu$]
\label{thm:bdry-inter}
Let $c$ satisfy \Bzero--\Bone\ and 
$\uu=\uu^{c^*c}$ be a $c$-convex function (which implies $\p^c\uu(\cl \U)=\cl \V$),
and $\lambda>0$.
\begin{enumerate}
\item[(a)]
If $|\p^c\uu| \geq \lambda$ on $X \subset \cl U$ and $\V$
is strongly $c^*$-convex with respect to $X$,
then interior points of $X$ cannot be mapped by $\p^c\uu$ to boundary points of $\V$:
i.e.\ $(X \times \p V) \cap \p^c \uu \subset (\p X \times \p \V)$.

\item[(b)] If $|\p^c\uu| \leq \Lambda$ on $\cl \U$, and $\U$ is
strongly $c$-convex with respect to $\V$, then boundary points of
$\U$ cannot be mapped by $\p^c \uu$ into interior points of $\V$:
i.e.\ $\partial \U \times \V$ is disjoint from  $\p^c \uu$.
\end{enumerate}
\end{theorem}

\begin{proof}
Note that when $X$ is open the conclusion of (a) implies $\p^c u$
is disjoint from $ X \times \p V$. We therefore remark that it
suffices to prove (a), since (b) follows from (a) exchanging the
role $x$ and $y$ and observing that $|\p^c\uu| \leq \Lambda$
implies $|\p^\cs\uu^\cs|\geq 1/\Lambda$ as in Lemma \ref{L:cMA
properties}(d).

Let us prove (a).
Fix any point $\tilde x$ in the interior of $X$, and
$\tilde y \in \p^c\uu(\tilde x)$. Assume by contradiction that
$\tilde y \in \p\V$.  At $(\tilde x,\tilde y)$ we use \Bzero--\Bone\ to
define cost-exponential coordinates $(p,q)\longmapsto (x(q),y(p))$
by
\begin{eqnarray*}
p =& - D_x c(\tilde x, y(p)) + D_x c(\tilde x, \tilde y)  &\in T^*_\tx(U)\\
q =& D^2_{xy} c(\tilde x,\tilde y)^{-1}(D_y c(x(q), \tilde y) - D_y c(\tilde x, \tilde y)) &\in T_\tx(U)
\end{eqnarray*}
and define a modified cost and potential by subtracting null
Lagrangian terms:
\begin{eqnarray*}
\tilde c(q,p) &:=& c(x(q),y(p)) - c(x(p),\tilde y) - c(\tilde x, y(p)) \\
\tilde \uu(q) &:=& \uu(x(q)) + c(x(q),\tilde y).
\end{eqnarray*}
Similarly to Corollary \ref{C:Jacobian transform},
$|\partial^{\tilde c} \tilde u| \ge \tilde \lambda
:=\lambda/(\gamma^+_c \gamma^-_c)$, where $\gamma^\pm_c$ denote
the Jacobian bounds \eqref{Jacobian bound} for the coordinate
change. Note $(\tilde x,\tilde y) = (x(\zero),y(\zero))$
corresponds to $(p,q) = (\zero,\zero)$. Since $c$-segments with
respect to $\tilde y$ correspond to line segments in $\U_{\tilde
y} := -D_y c(U,\ty)$ we see $D_p \tilde c(q,\zero)$ depends
linearly on $q$, whence $D^3_{qqp} \tilde c(q,\zero)=0$; similarly
$c^*$-segments with respect to $\tilde x$ become line segments in
the $p$ variables, $D_q \tilde c(\zero, p)$ depends linearly on
$p$, $D^3_{ppq} c(\zero, p)=0$, and the extra factor $D^2_{xy}
c(\tilde x, \tilde y)^{-1}$ in our definition of $x(q)$ makes $-
D^2_{pq}\tilde c(\zero,\zero)$ the identity matrix (whence $q= -
D_p \tilde c(\zero, q)$ and $p = -D_q \tilde c(p,\zero)$ for all
$q$ in $\U_\ty= x^{-1}(U)$ and $p$ in $\V_\tx := y^{-1}(V))$.
Although the change of variables $(q,p) \longmapsto (x(p),y(q))$ is only a $C^3$
diffeomorphism, we can still take four derivatives of the modified cost provided at least one of
the four derivatives is with respect to $q$ and another is with respect to $p$.
We denote $X_\ty := x^{-1}(X)$ 
and choose orthogonal coordinates on $U$ which make $-\hat e_n$
the outer unit normal to $\V_\tx \subset T^*_\tx \U$ at $\tilde p
=\zero$.
Note that 
$V_\tx$ is strongly convex by hypothesis~(a).

In these variables,  consider a small cone of height $\e$ and angle $\theta$
around the $-\hat e_n$ axis:
$$
E_{\theta,\e} := \left\{q \in \R^n \mid \Big| - \hat e_n -
\frac{q}{|q|}\Big| \le \theta, |q| \le \e \right\}
$$
Observe that, if $\theta,\e$ are small enough, then $E_{\theta,\e}
\subset  X_\ty$, and its measure is of order $\e^n\theta^{n-1}$.
Consider now a slight enlargement
$$
E'_{\theta,C_0\e}:=\Bigl\{p =(P,p_n) \in \R^n \mid p_n \le
\theta|p|+C_0\e|p|^2\Bigr\},
$$
of the polar dual cone, where $\e$ will be chosen sufficiently
small depending on the large parameter $C_0$ forced on us later.
\begin{figure}
\begin{center}
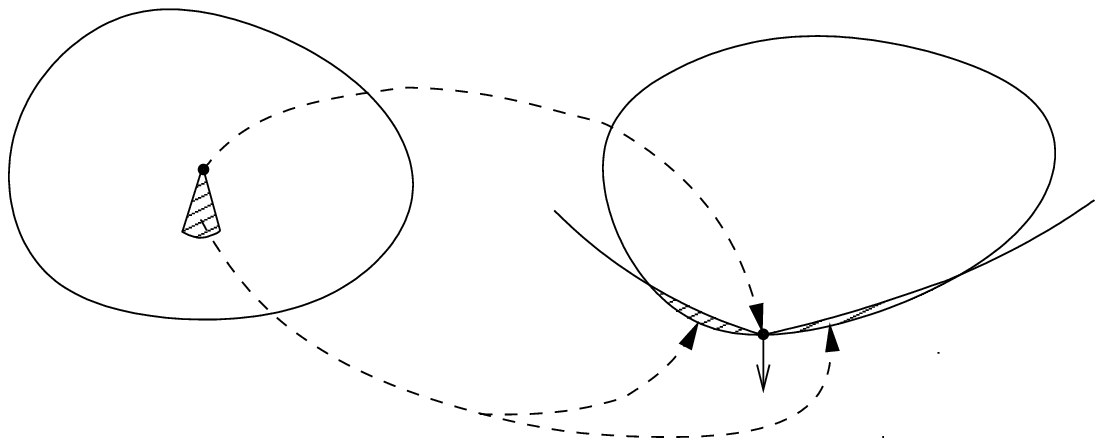\caption{{\small If $\p^\tc \tu$ sends an interior point onto a boundary point,
then by $\tc$-monotonicity of $\p^\tc \tu$ the small cone
$E_{\theta,\e}$ has to be sent onto $E_{\theta,C_0\e}' \cap
\V_{\tx}$. Since for $\e>0$ small but fixed
$\Leb{n}(E_{\theta,\e}) \sim \theta^{n-1}$, while
$\Leb{n}(E_{\theta,C_0\e}' \cap \V_{\tx}) \lesssim \theta^{n+1}$}
(by the uniform convexity of $\tilde \V_{\tx}$), we get a
contradiction as $\theta \to 0$.}\label{figcone}
\end{center}
\end{figure}

The strong convexity ensures $\V_\tx$ is contained in a ball
$B_R(R\hat e_n)$ of some radius $R>1$ contained in the half-space
$p_n \ge 0$ with boundary sphere passing through the origin. As
long as $C_0\e < (6 R)^{-1}$ we claim $E'_{\theta,C_0\e}$
intersects this ball --- a fortiori $\V_\tx$ --- in a set whose
volume tends to zero like $\theta^{n+1}$ as $\theta \to 0$.
Indeed,  from the inequality 
$$
p_n
\le \theta \sqrt{|P|^2 + p_n^2} + \frac{1}{6}|P| + \frac{1}{3}p_n
$$
satisfied by any $(P,p_n) \in E'_{\theta,C_0\e} \cap B_R(R \hat e_n)$ we deduce
$p_n^2 \le |P|^2 (1 + 9 \theta^2)/(2-9\theta^2)$,
i.e.\ $p_n < |P|$ if $\theta$ is small enough.  Combined with the
further inequalities
$$
\frac{|P|^2}{2R} \le p_n
\le \theta \sqrt{|P|^2 + p_n^2} + C_0 \e |P|^2 + C_0 \e p_n^2
$$
(the first inequality follows by the strong convexity of
$\V_\tx$),
this yields $|P| \le 6 \theta \sqrt{2}$ 
and $p_n \le O(\theta^2)$ as $\theta \to 0$. Thus $\Leb n
(E_{\theta, C_\e}' \cap \V_\tx) \le C \theta^{n+1}$ for a
dimension dependent constant $C$, provided $C_0\e < (6R)^{-1}$.

The contradiction now will come from the fact that, thanks to the
$\tilde c$-cyclical monotonicity of $\partial^{\tilde c} \tilde
\uu$, if we first choose $C_0$ big and then we take $\e$
sufficiently small, the image of all $q \in E_{\theta,\e}$ by
$\p^{\tilde c} \tilde \uu$ has to be contained in
$E'_{\theta,C_0\e}$ for $\theta$ small enough. Since $\p^{\tilde
c} \tilde \uu\Big(\cl{X_\ty}\Big) \subset \cl{\V_\tx}$ this will
imply
$$
\e^n\theta^{n-1} \sim \tilde \lambda \Leb{n} (E_{\theta,\e}) \leq
|\p^{\tilde c}{\tilde \uu}|(E_{\theta,\e}) \leq \Leb{n} (\V_\tx
\cap E'_{\theta,C_0\e} ) \leq C\theta^{n+1},
$$
which gives a contradiction as $\theta \to 0$, for $\e>0$ small but fixed.

Thus all we need to prove is that, if $C_0$ is big  enough, then $\p^{\tilde c} \tilde \uu(E_{\theta,\e}) \subset
E'_{\theta,C_0\e}$ for any $\e$ sufficiently small. Let
$q \in E_{\theta,\e}$ and $p \in \p^{\tilde c}\tilde \uu(q)$.
Combining
\begin{align*}
 \int_0^1 ds \int_0^1 dt \,D^2_{qp} \tilde c(sq,tp)[q,p]
=\tilde c(q,p) + \tilde c(\zero,\zero)- \tilde c(q,\zero)-c(\zero,p)
\le 0
\end{align*}
(where the last inequality is a consequence of $\tilde
c$-monotonicity of $\partial^{\tilde c}\tilde u$; see for instance
\cite[Definitions 5.1 and 5.7]{Villani09}) with
\begin{align*}
D^2_{qp} \tilde c(sq,tp)
= &D^2_{qp} \tilde c(\zero,tp) + \int_0^s ds' D^3_{qqp} \tilde c(s'q,tp)[q] \\
= &D^2_{qp} \tilde c(\zero,\zero) + \int_0^t dt' D^3_{qpp} \tilde c(\zero,t'p)[p] \\
& + \int_0^s ds' D^3_{qqp} \tilde c(s'q,\zero)[q] + \int_0^s ds' \int_0^t dt' D^4_{qqpp} \tilde c(s'q,t'p)[q,p]
\end{align*}
yields
\begin{align*}
- \langle q, p\rangle
&\le - \int_0^1 ds \int_0^1dt \int_0^s ds' \int_0^t dt' D^4_{qqpp} \tilde c(s' q,t'p)[q,q,p,p] \\
&\le C_0 |q|^2 |p|^2
\end{align*}
since $D^3_{qpp} \tilde c(\zero,t'p)$ and $D^3_{qqp} \tilde
c(s'q,\zero)$ vanish in our chosen coordinates and $-D^2_{pq}
\tilde c(\zero,\zero)$ is the identity matrix. Due to the
tensorial nature of the cross-curvature \eqref{MTW}, $C_0$ depends
on $\|c\|_{C^4(U \times V)}$ and the bi-Lipschitz constants
$\beta^\pm_c$ from \eqref{bi-Lipschitzbound}.

From the above inequality and the definition of $E_{\theta,\e}$ we deduce
$$
p_n = \langle p, \hat e_n + \frac{q}{|q|}\rangle - \langle p,
\frac{q}{|q|}\rangle \le \theta |p| + C_0 \e |p|^2
$$
so $p \in E'_{\theta,C_0\e}$ as desired.
\end{proof}

\section{The Monge-Amp\`ere measure 
dominates the $\tc$-Monge-Amp\`ere measure} 

In this section
we shall prove that --- up to constants --- the ordinary Monge-Amp\`ere measure $|\p \tu|$
dominates the $\tc$-Monge-Amp\`ere measure $|\p^\tc \tu|$, when defined in the coordinates
introduced in Theorem \ref{thm:apparently convex}. Let us begin
with a lemma which motivates our proposition heuristically. The
conclusions of the lemma extend easily from smooth to non-smooth
functions by an approximation argument combining Lemma \ref{L:cMA
properties}(b)--(c) with results of Trudinger and Wang
\cite{TrudingerWang08p}. However this approach would require the
domains $\U$ and $\V$ to be smooth,  so in Proposition \ref{P:c-ma
less ma} we prefer to construct an explicit approximation which
proves the statement we need, requires no additional smoothness
hypotheses, and is logically independent of both Lemma
\ref{lem:c-ma less ma} and \cite{TrudingerWang08p}.

\begin{lemma}\label{lem:c-ma less ma}
Assume \Bzero--\Bthree\ and
let $\tu:\cl U_\ty \longmapsto \R$ be a convex $\tc$-convex function
as in Theorem \ref{thm:apparently convex}.  If $\tu \in C^2(\U'_\ty)$
for some open set $U'_\ty \subset \U_\ty$, then
$|\p^\tc \tilde u| \le \gamma^-_\tc |\p \tilde u|$ on $\U'_\ty$, 
where $\gamma^\pm_\tc = \gamma^\pm_\tc(\U'_\ty \times \V)$ are defined as in (\ref{Jacobian bound}).
\end{lemma}

\begin{proof}
In addition to the convexity of $\tilde u(q)$,
for any $y \in \cl \V$ Theorem \ref{thm:apparently convex}
asserts the convexity of the $\tilde c$-convex function
$q \in \cl\U_\ty \longmapsto -\tc(q,y)$.  Thus
$$
\det (D^2_{qq} \tilde u(\tq) + D^2_{qq} \tc(\tq,y)) \le \det
D^2_{qq} \tilde u(\tq)
$$
by the concavity of $S \longmapsto \det^{1/n}(S)$ on symmetric
non-negative definite matrices.
On the other hand, for $\tilde c$-convex $\tilde u \in C^2(U'_\ty)$,  the measure
$|\p^\tc \tilde u|$ is absolutely continuous, with Lebesgue density
given by the left hand side of (\ref{Monge-Ampere type equation}).
Thus at any $\tilde q \in \U'_\ty$,
\begin{equation}
\label{eq:det} \frac{d |\p^{\tilde c} \tu|}{d \Leb{n}}(\tq) =
\frac{\det(D^2_{qq} \tilde u(\tq) + D^2_{qq} c(\tq,\tilde
G(\tq)))} {|\det D^2_{qy} \tc(y,\tilde G(\tq))|} \le \gamma^-_\tc
\det D^2_{qq} \tilde u(\tq) \end{equation}
as desired.
\end{proof}

We now prove the proposition that we actually need subsequently.

\begin{proposition}[Monge-Amp\`ere measure dominates $\tc$-Monge-Amp\`ere measure]
\label{P:c-ma less ma} Assume \Bzero--\Bthree,
and let $\tu:\cl U_\ty \longmapsto \R$
be a convex $\tc$-convex function from Theorem \ref{thm:apparently convex}.
Then
$|\p^\tc \tilde u| \le \gamma^-_\tc |\p \tilde u|$ on $\U'_\ty \subset \U_\ty$, 
where $\gamma^\pm_\tc = \gamma^\pm_\tc(\U'_\ty \times \V)$ are defined as in (\ref{Jacobian bound}).
\end{proposition}

\begin{proof}
It suffices to prove
$|\p^\tc \tilde u|\big(B_r(\bq)\big) \le \gamma^-_\tc |\p \tilde u|\big(\overline {B_r(\bq)}\big)$
for each ball whose closure is contained in $\U'_\ty$.  Given such a ball, let
$h(q) := |q -\bar q|$,  and let $\rho_\e(q) = \e^{-n}\rho(q/\e) \ge 0$ be a smooth mollifier
vanishing outside $B_\e(\zero)$ and carrying unit mass.
For $\e>0$ sufficiently small, we can define the smooth convex function
$$
\tu_{\e,\d}=(\tu+\d h)*\rho_\e
$$
on $\overline{B_r(\bq)}$.
Since $\tu$ and $h$ are locally Lipschitz, letting $R$ denote a
bound for ${\rm Lip}(\tu)+{\rm Lip}(h)$ inside $B_r(\bq)$ yields
$$
\|\tu_{\e,\d} -(\tu+\d h)\|_{L^\i(B_r(\bq))} \leq \e R
$$
for all $\d<1$.

\textit{Claim:} Fix $0<t < 1$ and $0<\d<1$.
For $\e>0$ sufficiently small, we claim that
to each $y_0 \in \p^\tc \tu(B_{t r}(\bq))$ corresponds some
$q_{\e,\d} \in B_{r}(\bq)$ such that $(q_{\e,\d},y_0) \in \p^\tc \tu_{\e,\d}$.

Indeed, for $y_0 \in \p^\tc\tu(q_0)$ with $|q_0-\bq|< t r$, observe
$$
\tu(q) \geq -c(q,y_0) -\tu^{\tc^*}(y_0) \qquad \forall\, q \in \cl U_\ty,
$$
with equality at $q_0$. Moreover $h(q_0) \leq t r$. Therefore
\begin{align*}
\tu_{\e,\d}(q_0)-(-c(q_0,y_0)-\tu^{\tc^*}(y_0))
&\leq \e R +\tu(q_0)+\d h(q_0)-(-c(q_0,y_0)-\tu^{\tc^*}(y_0))\\
&=\e R+\d h(q_0) \leq \e R + t r \d.
\end{align*}
On the other hand, if $q \in \p B_{r}(\bq)$, then $h(q) = r$, and so
\begin{align*}
\tu_{\e,\d}(q)-(-c(q,y_0) -\tu^{\tc^*}(y_0)) &\geq -\e R
+\tu(q)+\d h(q)-(-c(q,y_0) -\tu^{\tc^*}(y_0))\\ &
\geq -\e R+\d h(q) = -\e R + r \d.
\end{align*}
Thus, if for fixed $\delta$ small we choose $\e$ small enough so
that
$$
r \d-\e R > t r \d+\e R,
$$
we deduce that if we lower the graph of the function
$-c(q,y_0)-\tu^{\tc^*}(y_0)$ to the lowest level at which it
intersects the graph of $\tu_{\e,\d}$, then the point of
intersection must lie over $B_{r}(\bq)$. This proves the claim.

Having established the claim,
let $E \subset B_{r} (\bq)$ denote the (Borel) set of all
$q_{\e,\d}$ which arise from $\p \tu_{\e,\d}(B_{t r}(\bq))$ in this way.
Since $\tu_{\e,\d}$ is smooth, 
the condition $y_0 \in \p^\tc\tu_{\e,\d}(q_{\e,\d})$ implies
\begin{equation}
\label{subdiff1} D_q\tu_{\e,\d}(q_{\e,\d})=-D_q \tc(q_{\e,\d},y_0),
\end{equation}
as well as
\begin{equation}
\label{subdiff2} D_{qq}\tu_{\e,\d}(q_{\e,\d}) \geq -D_{qq} \tc(q_{\e,\d},y_0).
\end{equation}
By \eqref{subdiff1} and \Bone\, we can define a smooth map
$G_{\e,\d}(q_0)$ throughout $E$ using the relation
$$
D_q\tu_{\e,\d}(q_0)=-D_q \tc(q_0,G_{\e,\d}(q_0)),
$$
and find that $\p^\tc\tu_{\e,\d}(q_{\e,\d}) = \{G_{\e,\d}(q_{\e,\d})\}$ is a singleton.
In this way we obtain
\begin{align*}
|\p^c\tu|(B_{t r}(\bq)) \leq |\p^c\tu_{\e,\d}|(E) &=\int_E |\det D_q G_{\e,\d}|(q)\,dq\\
&=\int_E \frac{\det (D_{qq}\tu_{\e,\d} (q) +
D_{qq}\tc(q,G_{\e,\d}(q)))}{|\det D^2_{qy} \tc(q,
G_{\e,\d}(q))|}\,dq,
\end{align*}
where in the last equality we used \eqref{subdiff2} to deduce that
$D_{qq}\tu_{\e,\d}(q) +D_{qq} c(q,G_{\e,\d}(q))$ is non-negative
definite. Hence the inequality
$$
\frac{\det (D_{qq}\tu_{\e,\d} (q) +
D_{qq}\tc(q,G_{\e,\d}(q)))}{|\det D^2_{qy} \tc(q, G_{\e,\d}(q))|}
\leq\g_\tc^- \det D_{qq}\tu_{\e,\d}(q)
$$
holds (similarly to \eqref{eq:det} above), and so
$$
|\p^c\tu|(B_{t r}(\bq)) \leq \g_\tc^-\int_E \det (D_{qq}\tu_{\e,\d} (q))\,dq
 \leq \g_\tc^-|\p\tu_{\e,\d}|(B_{r}(\bq)).
$$
Letting first $\e \to 0$ and then $\d \to 0$, we finally deduce
$$
|\p^c\tu|(B_{t r}(\bq)) \leq \limsup_{\e,\d \to 0}
\g_\tc^-|\p\tu_{\e,\d}|(B_{r}(\bq)) \leq \g_\tc^-|\p\tu|\big( \overline{B_{r}(\bq)}\big).
$$
Here, to see the last inequality one may, for instance, use
Lemma~\ref{L:cMA properties}(b) with $c(x,y) = -\<x, y\>$.
Arbitrariness of $0<t<1$ yields the desired result.

\end{proof}

\section{Alexandrov type estimates and affine renormalization}\label{S:Alexandrov}
In this section we prove the key estimates for $c$-convex
potential functions which will eventually lead to the continuity
and injectivity of optimal maps. Namely, we extend Alexandrov type
estimates commonly used in the analysis of Monge-Amp\`ere
equations (thus for the cost $c(x,y)= -\<x, y\>$), to general
non-negatively curved cost functions.  This is established in
Lemma~\ref{lemma-infgeq} (plus Proposition \ref{P:c-ma less ma})
and Lemma~\ref{L:prop-infleq}. These estimates are used to compare
the infimum of $c$-convex function on a section with the size of
the section, which are the key ingredients in the proof of our
main results; see Propositions~\ref{prop:alex1} and
\ref{prop:estimate}. Lemma~\ref{L:prop-infleq} represents the most
nontrivial and technical result we obtain in this section.

We recall a basic lemma for convex sets due to Fritz John \cite{John48},
which will play an essential role in the rest of the paper.

\begin{lemma}[John's lemma]\label{L:John}
For a compact convex set $S \subset \R^n$, 
there exists an affine transformation $ L: \R^n \to \R^n$ such
that $\cl {B_1} \subset L^{-1}(S) \subset \cl {B_{n}}$.
\end{lemma}

We now estimate the infimum of $\tu$ in terms of the
Monge-Amp\`ere measure in a section. The following lemma is a standard fact
for convex functions. With Lemma~\ref{L:John} in mind, we state it for normalized functions
$\tu^*$ and sections $Z^*$.  However, the estimate (\ref{inf geq}) is invariant under the affine
renormalization \eqref{renormalized solution}; according to \eqref{compare p var with p var *},
it holds with or without stars.

\begin{lemma}[Upper bound on Dirichlet solutions to Monge-Amp\`ere inequalities]
\label{lemma-infgeq} Let $\tu^*:\R^n \longmapsto \R \cup \{+\infty \}$ be a
convex function whose section $Z^*:=\{ \tu^* \le 0 \}$ satisfies
$B_1 \subset Z^* \subset \cl B_n$. Assume that $\tu^*=0$ on $\p Z^*$.
Then, for all $t \in (0,1)$,
\begin{equation}\label{inf geq}
|\p\tu^*|\left(tZ^*\right) \leq \frac{C(n)}{(1-t)^n}\frac
{|\inf_{Z^*}\tu^*|^n}{\Leb{n}(Z^*)},
\end{equation}
where $tZ^*$ denotes the dilation of $Z^*$ by a factor $t$ with
respect to the origin.
\end{lemma}
Although the proof of this result is classical (see for instance
\cite{Gutierrez01}), for sake of completeness we prefer to give
all the details.
\begin{proof}
We can assume that $\tu^*|_{Z^*} \not\equiv 0$, otherwise the
estimate is trivial.
It is not difficult to prove that
\begin{equation} \label{bound_subdiff_ball} |p^*|
\leq \frac{|\inf_{Z^*} \tu^*|}{(1-t)} \qquad
\forall\,p^* \in \p\tu^*(tZ^*).
\end{equation}
Indeed, if $q^* \in tZ^*$ and $p^* \in \p\tu^*(q^*)$,
then
$$
\<p^*, q-q^*\> \leq \tu(q)-\tu(q^*)=|\tu(q^*)| \qquad \forall\,q
\in \p Z^*,
$$
and taking the supremum in the left hand side among all $q \in \p Z^*$
\eqref{bound_subdiff_ball} follows.
Thus, since $\Leb{n}(Z^*) \le C(n)$, we conclude
$$
|\p\tu^*|(tZ^*)\leq
\frac{C(n)}{(1-t)^n}\frac
{|\inf_{Z^*}\tu^*|^n}{\Leb{n}(Z^*)}
.
$$
\end{proof}
Combining the above lemmas, we obtain:

\begin{proposition}
\label{prop:alex1}
Assume \Bzero--\Bthree\ and define $\gamma^\pm_\tc=\gamma^\pm_\tc(Z \times \V)$ as in (\ref{Jacobian bound}).
Any convex $\tilde c$-convex function  $\tu:\cl \U_\ty \longmapsto \R$ from Theorem \ref{thm:apparently convex}
which satisfies $|\p^c \tu| \in [\l, 1/\l]$ in a section of the form
$
Z:=\{ q  \in \cl \U_\ty | \ \tu(q) \le 0 \},
$
and $\tu = 0$ on $\partial Z$, also satisfies
\begin{equation}\label{|Z| less inf varphi}
\Leb{n}(Z)^2 \leq C(n)
\frac{\gamma^-_\tc}{\l} |\inf_{Z} \tu|^n.
\end{equation}
\end{proposition}

\begin{proof}
First use the affine map $L$ as given in Lemma~\ref{L:John} to renormalize $\tu$ into $\tu^*$
using \eqref{renormalized solution}.  This does not change the bound $|\p \tu^*| \in [\l,1/\l]$,
but allows us to apply Lemma~\ref{lemma-infgeq} with $t=1/2$.
Its conclusion (\ref{inf geq}) has been expressed in a form which holds
with or without the stars, in view of \eqref{compare p var with p var *}.
Proposition~\ref{P:c-ma less ma} now yields the desired inequality \eqref{|Z| less inf varphi}.
\end{proof}

\subsection{$\tc$-cones over convex sets}

We now progress toward the Alexandrov type estimate in
Lemma~\ref{L:prop-infleq}. In this subsection we construct and
study the $\tc$-cone associated to the section of a $\tc$-convex
function. This $\tc$-cone plays an essential role in our proof of
Lemma~\ref{L:prop-infleq}.

\begin{definition}[$\tc$-cone]
\label{def:ccone}
Assume \Bzero--\Btwo\ and \Athreew, 
and let $\tu:\cl \U_\ty \longmapsto \R$ be the $\tilde c$-convex function with
convex level sets from Theorem~\ref{thm:apparently convex}.
Let $Z$ denote the section $\{ \tilde u \le 0\}$, fix $\tilde q
\in \intr Z$, and assume $\tu = 0$ on $\p Z$. The \emph{$\tc$-cone
$h^\tc: \U_\ty \longmapsto \R$ generated by $\tq$ and $Z$ with
height $-\tu(\tq)>0$} is given by
\begin{align}\label{c-cone}
h^\tc(q)  :=\sup_{y \in \cl \V}
\{ -\tc(q, y) + \tc(\tq, y) + \tu(\tq) \mid
   -\tc(q, y) + \tc(\tilde q, y) +\tu(\tq)\le 0 \text{ on } \p Z \}.
\end{align}
\end{definition}

Notice the $\tc$-cone $h^\tc$ depends only on the convex set $Z \subset \cl \U_\ty$, $\tq \in \intr Z$,
and the value $\tu(\tq)$, but is otherwise independent of $\tilde u$.
Recalling that $\tc(q, \ty) \equiv 0$ on $\U_{\ty}$, we record several key properties
of the $\tc$-cone:

\begin{lemma}[Basic properties of $\tc$-cones]
\label{lemma:propccone}
Adopting the notation and hypotheses of Definition \ref{def:ccone},
let $h^\tc:\cl \U_\tq \longmapsto \R$ be the $\tc$-cone generated by $\tilde q$ and
$Z$ with height $-\tu (\tilde q)$. Then
\begin{enumerate}
\item[(a)] $h^\tc$ has convex level sets; it is a convex function if \Bthree\ holds;
\item[(b)] $h^\tc(q) \geq h^\tc(\tilde q) = \tu(\tilde q)$ for all $q \in Z$;
\item[(c)] $h^\tc = 0 \text{ on } \p Z$;
\item[(d)] $\p^\tc h^\tc (\tilde q) \subset \p^\tc \tu (Z)$.
\end{enumerate}
\end{lemma}

\begin{proof}
Property (a) is a consequence of the level-set convexity of $q
\longmapsto -\tc(q, y)$ proved in Theorem \ref{thm:apparently
convex}, or its convexity assuming \Bthree. Moreover, since
$-\tc(q,\ty)+\tc(\tilde q,\ty)+\tu(\tilde q)=\tu(\tilde q)$ for
all $q\in \U_\ty$, (b) follows. For each pair $ z \in \p Z$ and
$y_z \in \p^\tc \tu (z)$, consider the supporting mountain $m_z
(q)= -\tc(q, y_z ) + \tc(z, y_z) $, i.e. $m_z(z)=0=\tu (z)$ and
$m_z \le \tu$. Consider the $\tc$-segment $\sigma(t)$ connecting
$\sigma(0)=\ty$ and $\sigma(1)= y_z$ in $\V$ with respect to $z$.
Since $-\tc(q, \ty) \equiv 0$, by continuity there exists some $t
 \in ]0, 1]$ for which $\bar m_z (q):= -\tc(q, \sigma(t)) +
\tc(z, \sigma(t))$ satisfies $\bar m_z (\tilde q) = \tu(\tilde
q)$. From Loeper's maximum principle (Theorem \ref{T:DASM} above),
we have
$$
\bar m_z \le \max[ m_z , -\tc(\cdot, \ty)]= \max[m_z , 0],
$$
and therefore, from $m_z \le \tu$,
$$
\bar m_z \le 0 \text{ on } Z.
$$
By the construction, $\bar m_z$ is of the form
$$
-\tc(\cdot, y) + \tc(\tilde q, y) + \tu (\tilde q),
$$
and vanishes at $z$. This proves (c).
Finally (d) follows from (c)
and  the fact that $h^\tc (\tilde q ) = \tu(\tilde q)$.
Indeed, it suffices to move down the supporting mountain of $h^\tc$ at $\tilde q$
until the last moment at which it touches the graph of $\tu$ on $Z$ from below.
The conclusion then follows from Loeper's local to global principle,
Corollary \ref{C:local-global} above.
\end{proof}

The following estimate shows that the Monge-Amp\`ere measure, and
the relative location of the vertex within the section which
generates it, control the height of any well-localized $\tc$-cone.
Afficionados of the Monge-Amp\`ere theory may be less surprised by
this estimate once it is recognized that the localization in
coordinates ensures the cost is approximately affine, at least in
one of its two variables. Still, it is vital that the
approximation be controlled! Together with
Lemma~\ref{lemma:propccone}(d), this proposition plays a key role
in the proof of our Alexandrov type estimate
(Lemma~\ref{L:prop-infleq}).

\begin{proposition}
[Lower bound on the Monge-Amp\`ere measure of a small $\tc$-cone]
\label{P:partial c-cone} Assume \Bzero--\Bthree\ and define
$\tilde c \in C^3\big(\cl\U_\ty \times \cl\V\big)$ as in
Definition \ref{D:cost exponential}. Let  $Z \subset \cl \U_\ty$
be a closed convex set and $h^\tc$ the $\tc$-cone generated by
$\tq \in \intr Z$ of height $-h^\tc(\tq)>0$ over $Z$. Let
$\Pi^+,\Pi^-$ be two parallel hyperplanes contained in $T^*_\ty \V
\setminus Z$ and touching $\p Z$ from two opposite sides. Then
there exists $\e_c>0$ small, depending only on the cost (and given
by Lemma~\ref{L:c estimate}), and a constant $C(n)>0$ depending
only on dimension, such that if $\diam(Z) \leq \e_c/C(n)$ then
\begin{align}\label{p c-cone lower}
|h^\tc (\tilde q)|^n \leq C(n) \frac{\min\{\dist(\tilde
q,\Pi^+),\dist(\tilde q,\Pi^-)\}}{\ell_{\Pi^+}}|\p h^\tc|(\{\tilde
q\}) \Leb{n}(Z),
\end{align}
where $\ell_{\Pi^+}$ denotes the maximal length among all
the segments obtained by intersecting $Z$ with a line orthogonal
to $\Pi^+$.
\end{proposition}

\noindent To prove this, we first observe a basic estimate on the  cost function $c$.
\begin{lemma}\label{L:c estimate}
Assume \Bzero--\Btwo. For $\tilde c \in C^3\big(\cl \U_\ty \times \cl\V\big)$ from Definition
\ref{D:cost exponential} and each $y \in \cl \V$ and $q,\tilde q \in \cl \U_\ty$,
\begin{align}\label{slope compare}
|-D_q \tc(q,y)+D_q\tc(\tilde q,y)|
&\le  \frac{1}{\e_c}  | q - \tilde q|\,|D_q \tc(\tilde q,y)|
\end{align}
where $\e_c$ is given by
$ 
\e_c^{-1} = 2 (\beta^+_c)^4 (\beta^-_c)^6 \| D^3_{xxy}c \|_{L^\infty(\U \times \V)}
$ 
in the notation \eqref{bi-Lipschitzbound}.
\end{lemma}

\begin{proof}
For fixed $\tq \in \cl \U_\ty$ introduce the $\tilde
c$-exponential coordinates $p(y) = - D_q \tilde c(\tq, y)$. The
bi-Lipschitz constants \eqref{bi-Lipschitzbound} of this
coordinate change are estimated by $\beta^\pm_\tc \le \beta^+_c
\beta^-_c$ as in Corollary~\ref{C:Jacobian transform}.
Thus
\begin{align*}
\dist(y,\ty)
&\leq \beta^-_\tc |-D_q \tc(\tq, y) + D_q \tc(\tq, \ty )|\\
&=   \beta_c^+ \beta_c^-  |D_q \tc(\tilde q, y)|.
\end{align*}
where $\tc(q,\ty) \equiv 0$ from Definition \ref{D:cost
exponential} has been used. Similarly, noting the convexity \Btwo\
of $\V_\tq :=p(\V)$,
\begin{align*}
|-D_q \tc(\tilde q,y)+D_q\tc(q,y)|
&=|-D_q\tc(\tilde q,y)+D_q\tc(q,y)+D_q\tc(\tilde q,\ty)-D_q\tc(q,\ty)|
\\&\leq \|D^2_{qq} D_p \tc\|_{L^\i(\U_\ty \times \tilde \V_\tq)}  |\tq - q| |p(y)-p(\ty)|
\\&\leq \|D^2_{qq} D_y \tc\|_{L^\i(\U_\ty \times \V)} (\beta^-_c \beta^+_c)^2|\tq - q| \dist(y,\ty)
\end{align*}
The result follows since $|D^2_{qq} D_y \tc| \le( (\beta^-_c)^2 +
\beta^+_c (\beta^-_c)^3))  |D^2_{xx} D_y c| \le 2 \beta^+_c
(\beta^-_c)^3 |D^2_{xx} D_y c| $ . (The last inequality
follows from $\beta^{+}_{c}\beta^-_c \ge 1$.)
\end{proof}

\begin{proof}[Proof of Proposition~\ref{P:partial c-cone}]
We fix $\tilde q \in Z$.
Let $\Pi^i$, $i=1, \cdots n$, (with $\Pi^1$ equal either $\Pi^+ $ or $\Pi^-$)  be hyperplanes contained in $\R^n
\setminus Z$, touching $\p Z$, and such that
$\{\Pi^+,\Pi^2,\ldots,\Pi^n\}$ are all mutually orthogonal (so
that also $\{\Pi^-,\Pi^2,\ldots,\Pi^n\}$ are mutually
orthogonal).
Moreover we choose $\{\Pi^2,\ldots,\Pi^n\}$ in such a way that,
if $\pi^1(Z)$ denotes the projection of $Z$ on $\Pi^1$
and $\Haus{n-1}(\pi^1(Z))$ denotes its $(n-1)$-dimensional Hausdorff measure, then
\begin{equation}
\label{eq:choice hyperpl}
C(n)\Haus{n-1}(\pi^1(Z))\geq \prod_{i=2}^n \dist(\tilde
q,\Pi^i),
\end{equation}
for some universal constant $C(n)$. Indeed, as $\pi^1(Z)$ is
convex, by Lemma~\ref{L:John} we can find an ellipsoid $E$ such
that $E \subset \pi^1(Z) \subset (n-1)E$, and for instance we can
choose $\{\Pi^2,\ldots,\Pi^n\}$ among the hyperplanes orthogonal
to the axes of the ellipsoid (for each axis we have two possible
hyperplanes, and we can always choose the furthest one so that
\eqref{eq:choice hyperpl} holds).

Each hyperplane $\Pi^i $ touches $Z$ from outside, say at $q^i \in
T^*_\ty \V$. Let $p_i \in T_\ty V$ be the outward (from $Z$) unit
vector at $q^i$ orthogonal to $\Pi^i$. Then $s_i p_i \in \p h^\tc
(q^i)$ for some $s_i >0$, and by Corollary \ref{C:local-global}
there exists $y_i \in \p^\tc h^\tc(q^i)$ such that $$ -D_q \tc(
q^i , y_i) = s_i p_i .
$$
Define $y_i(t)$ as
$$
-D_q \tc(q^i, y_i (t) ) = t\, s_i p_i,
$$
i.e. $y_i(t)$ is the $\tc$-segment from $\ty$ to $y_i$ with respect to $q^i$.
As in the proof of Lemma~\ref{lemma:propccone} (c), the intermediate value theorem yields 
$0 <t_i \le 1$ such that
$$
-\tc(\cdot , y_i (t_i) ) + \tc(\tilde q, y_i (t_i)) + h^\tc (\tilde q) \le 0 \text{ on $Z$}
$$
with equality at $q^i$. Thus, by the definition of $h^\tc$, $y_i (t_i) \in \p^\tc h^\tc (\tilde  q) \cap \p^\tc h^\tc (q^i)$,
$$
-D_q \tc (\tilde q, y_i (t_i)) \in \p h^\tc (\tilde q)
\quad \text{and} \quad t_i s_i p_i =-D_q \tc (q^i, y_i (t_i)) \in \p h^\tc (q^i).
$$
\begin{figure}
\begin{center}
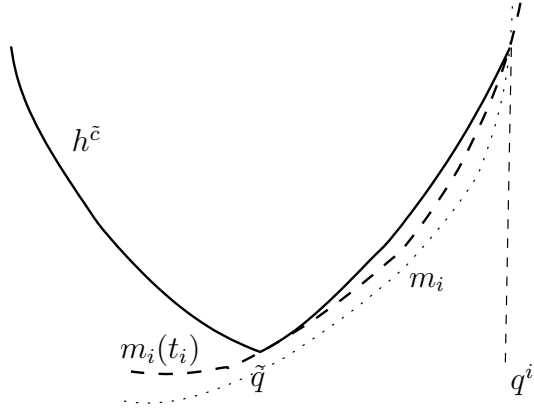\caption{{\small The dotted line represents the graph of $m_i:=-\tc(\cdot , y_i ) + \tc(\tilde q, y_i) + h^\tc (\tilde q)$,
while the dashed one represents the graph of $m_i(t_i):=-\tc(\cdot
, y_i(t_i) ) + \tc(\tilde q, y_i(t_i)) + h^\tc (\tilde q)$. The
idea is that, whenever we have $m_i$ a supporting function for
$h^\tc$ at a point $q^i \in \p Z$, we can let $y$ vary
continuously along the $\tc$-segment from $\ty$ to $y_i$ with
respect to $q^i$, to obtain a supporting function $m_i(t_i)$ which
touches $h^\tc$ also at $\ty$.}}\label{figc-cone}
\end{center}
\end{figure}
Therefore by the convexity of $h^\tc$ shown in Lemma
\ref{lemma:propccone}(a), the affine function $P^i$ with slope
$-D_q \tc (q^i, y_i (t_i))$ and with $P^i (\Pi^i) \equiv 0$
satisfies $P^i (\tilde q ) \le h^\tc (\tilde q)$. This shows
\begin{align}\label{D c at z i lower}
| -D_q \tc (q^i, y_i (t_i))| \ge \frac{|h^\tc  (\tilde q)|}{\dist (\tilde q, \Pi^i)}.
\end{align}
Also, by \eqref{slope compare}
\begin{align*}
|-D_q \tc(\tilde q, y_i (t_i) ) + D_q \tc(q^i , y_i (t_i)) |
& \le \frac{1}{\e_c} |\tilde q - q^i | \, |-D_q \tc(\tq , y_i (t_i)|\\
&\le  \frac{1}{\e_c} {\rm diam}\, Z \, |-D_q \tc(\tq , y_i (t_i)|.
\end{align*}
Therefore if ${\rm diam} \, Z \le \d_n\e_c $ with $\d_n>0$ small, 
each vector
$-D_q \tc( \tilde q, y_i(t_i))$ is close to $-D_q \tc(q^i, y_i (t_i))$, say
$$
|-D_q \tc(\tilde q, y_i (t_i) ) + D_q \tc(q^i , y_i (t_i)) | \leq \delta_n|-D_q \tc(\tq , y_i (t_i))|.
$$
Since the vectors $\{-D_q \tc(q^i, y_i (t_i))\}$ are mutually orthogonal, the above estimate implies that
for $\d_n$ small enough the convex hull of $\{-D_q \tc(\tilde q, y_i (t_i))\}$ has measure of order
$\prod_{i=1}^n|-D_q \tc(q^i , y_i (t_i))|$.
Thus, by the lower bound
\eqref{D c at z i lower} and the convexity of $\p h^\tc (\tilde q)$, we obtain
\begin{align*}
\Leb{n}(\p h^\tc (\tilde q)) & \ge C(n) \frac{|h^\tc (\tilde q)|^n}{\prod_{i=1}^n \dist(\tilde q, \Pi^i)}.
\end{align*}
Since $\Pi^1$ was either $\Pi^+$ or $\Pi^-$, we have proved that
$$
|h^\tc(\tilde q)|^n \leq C(n) \min\{\dist(\tilde
q,\Pi^+),\dist(\tilde q,\Pi^-)\} \prod_{i=2}^n \dist(\tilde
q,\Pi^i)|\p h^\tc|(\{\tilde q\}).
$$
To conclude the proof, we apply Lemma~\ref{lemma:orthogonal sections} below
with $Z'$ given by the segment obtained intersecting
$Z$ with a line orthogonal to $\Pi^+$. 
Combining
that lemma with \eqref{eq:choice hyperpl}, we obtain
$$
C(n)|Z| \geq \ell_{\Pi^+} \prod_{i=2}^n \dist(\tilde q,\Pi^i),
$$
and last two inequalities prove the proposition (taking $C(n) \ge 1/\delta_n$ larger if necessary).
\end{proof}

\begin{lemma}[Estimating a convex volume using one slice and an orthogonal projection]
\label{lemma:orthogonal sections}
Let $Z$ be a convex set in $\R^n = \R^{n'} \times \R^{n''}$.  Let
$\pi', \pi''$ denote the projections to the components $\R^{n'}$,
$\R^{n''} $, respectively.
Let $Z'$ be a slice orthogonal to the second component, that is
$$
Z' = (\pi'')^{-1}(\bar x'')\cap  Z \qquad\text{for some }\bar x'' \in \pi''(Z).
$$
Then there exists a constant $C(n)$, depending only on $n=n'+n''$,
such that
$$
C(n) \Leb{n}(Z) \ge \Haus{n'}(Z') \Haus{n''}(\pi'' (Z)),
$$
where $\Haus{d}$ denotes the $d$-dimensional Hausdorff measure.
\end{lemma}

\begin{figure}
\begin{center}
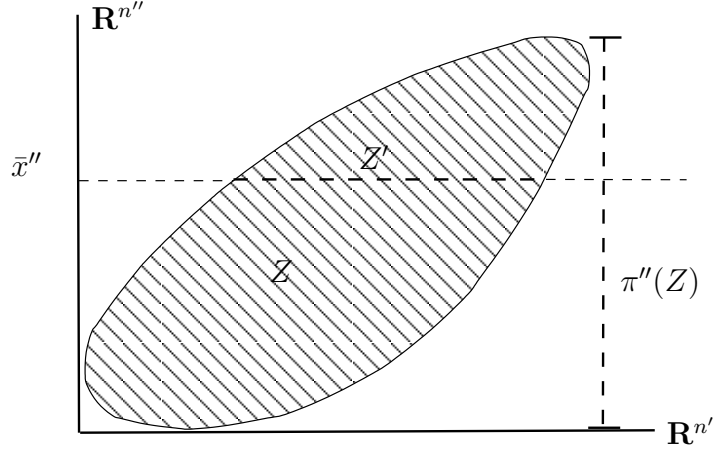\caption{{\small The volume of any convex set always controls the
product (measure of one slice) $\cdot$ (measure of the projection
orthogonal to the slice).}}\label{figconvex}
\end{center}
\end{figure}

\begin{proof}
Let $L:\R^{n''} \to \R^{n''}$ be an affine map with determinant $1$ given by
Lemma~\ref{L:John} such that $B_r \subset L(\pi''(Z)) \subset B_{n''r}$ for some $r>0$.
Then, if we extend $L$ to the whole $\R^n$ as $\tilde L(x',x'')=(x',Lx'')$, we have
$
\Leb{n}(L(Z))=\Leb{n}(Z)$, $\Haus{n'}(\tilde L(Z'))=\Haus{n'}(Z')$, and
$$
\Haus{n''}(\pi'' (\tilde L(Z)))=\Haus{n''}(L(\pi'' (Z)))=\Haus{n''}(\pi'' (Z)).
$$
Hence, we can assume from the beginning that $B_r \subset \pi''(Z)
\subset B_{n''r}$. Let us now consider the point $\bar x''$, and
we fix an orthonormal basis $\{\hat e_1,\ldots, \hat e_{n''}\}$ in
$\R^{n''}$ such that $\bar x''=c \hat e_1$ for some $c\leq 0$.
Since $\{r \hat e_1,\ldots,r \hat e_{n''}\} \subset \pi''(Z)$,
there exist points $\{x_1,\ldots,x_{n''}\} \subset Z$ such that
$\pi''(x_i)=r \hat e_i$. Let $C'$ denote the convex hull of $Z'$
with $x_1$, and let $V'$ denote the $(n'+1)$-dimensional strip
obtained taking the convex hull of $\R^{n'}\times\{\bar x''\}$
with $x_1$. Observe that $C'\subset V'$, and so
\begin{equation}
\label{eq:C'}
\Haus{n'+1}(C')=\frac{1}{n'+1}\dist(x_1,\R^{n'}\times\{\bar x''\})\Haus{n'}(Z') \geq \frac{r}{n'+1}\Haus{n'}(Z').
\end{equation}
We now remark that, since $\pi''(x_i)=r \hat e_i$ and $\hat e_i
\perp V'$ for $i=2,\ldots,n''$, we have $\dist(x_i,V')=r$ for all
$i=2,\ldots,n''$. Moreover, if $y_i\in V'$ denotes the closest
point to $x_i$, then the segments joining $x_i$ to $y_i$ parallels
$\hat e_i$, hence these segments are all mutually orthogonal, and
they are all orthogonal to $V'$ too. From this fact it is easy to
see that, if we define the convex hull
$$
C:=\co(x_2,\ldots,x_{n''},C'),
$$
then, since $|x_i-y_i|=r$ for $i=2,\ldots,n''$, by \eqref{eq:C'} and the inclusion $\pi''(Z) \subset B_{n''r}\subset \R^{n''}$ we get
$$
\Leb{n}(C)= \frac{(n'+1)!}{n!}\Haus{n'+1}(C') r^{n''-1} \geq
\frac{n'!}{n!}\Haus{n'}(Z') r^{n''} \geq
C(n)\Haus{n'}(Z')\Haus{n''}(\pi'' (Z)).
$$
This concludes the proof, as $C \subset Z$.
\end{proof}

\subsection{An Alexandrov type estimate}\label{SS:alex}

The next Alexandrov type lemma
holds for localized sections $Z$ of $\tc$-convex functions.

\begin{lemma}[Alexandrov type estimate and lower barrier]
\label{L:prop-infleq}
Assume \Bzero--\Bthree\ and
let $\tu:\cl \U_\ty \longmapsto \R$ be a convex $\tilde c$-convex function
from Theorem \ref{thm:apparently convex}.
Let $Z$ denote the section $\{ \tilde u \le 0\}$, assume $\tu =0$ on $\partial Z$,
and fix $\tilde q \in \intr Z$.
Let $\Pi^+,\Pi^-$ be two parallel hyperplanes contained in $\R^n
\setminus Z$ and touching $\p Z$ from two opposite sides. Then
there exists $\e_c'(n) >0$ small, depending only on dimension and
the cost function (with $\e_c'(n)= \e_c/C(n)$ given by Proposition
\ref{P:partial c-cone})
such that if $\diam(Z) \leq \e'_{c}(n)$ then
$$
|\tu(\tilde q)|^n \leq C(n) \gamma^+_\tc(Z \times \V)
\frac{\min\{\dist(\tilde q,\Pi^+),\dist(\tilde
q,\Pi^-)\}}{\ell_{\Pi^+}}|\p^\tc\tu|(Z) \Leb{n}(Z),
$$
where $\ell_{\Pi^+}$ denotes the maximal length among all
the segments obtained by intersecting $Z$ with a line orthogonal
to $\Pi^+$,
and $\gamma^\pm_\tc = \gamma^+_\tc(Z \times \V)$ is defined as in (\ref{Jacobian bound}).
\end{lemma}

\begin{proof}
 Fix $\tilde q \in Z$. Observe that $\tu = 0$ on $\p Z$ and consider the $\tc$-cone $h^\tc$
generated by $\tilde q$ and $Z$ of height $-h^\tc(\tq)=-\tu(\tq)$ as in \eqref{c-cone}.
From Lemma~\ref{lemma:propccone}(d) we have
$$
|\p^\tc h^\tc|(\{\tilde q\}) \le |\p^\tc\tu|(Z),
$$
\noindent
and from Loeper's local to global principle, Corollary \ref{C:local-global} above,
$$
\p h^\tc(\tq)=-D_q \tc(\tq,\p^\tc h^\tc(\tq)).
$$
Therefore
$$
|\p h^\tc|(\{\tq\}) \le \| \det D^2_{qy} \tc\|_{C^0 (\{\tq\} \times \V)} |\p^ch^c|(\{\tq\}).
$$
The lower bound on $|\p h^\tc|(\{\tilde q\})$ comes from \eqref{p c-cone lower}. This finishes the proof.
\end{proof}

\subsection{Estimating solutions to the $\tc$-Monge-Amp\`ere inequality 
$|\p^\tc \tu| \in [ \lambda, 1/\lambda]$}

Combining the results of Proposition~\ref{prop:alex1} and Lemma \ref{L:prop-infleq} yields:

\begin{proposition}[Bounding local Dirichlet solutions to $\tc$-Monge-Amp\`ere inequalities]
\label{prop:estimate} Assume \Bzero--\Bthree\ and let $\tu:\cl
\U_\ty \longmapsto \R$ be a convex $\tilde c$-convex function from
Theorem \ref{thm:apparently convex}. There exists $\e_c'(n) >0$
small, depending only on dimension and the cost function (and
given by Lemma~\ref{L:prop-infleq}), and constants $C(n)$,
$C_i(n)>0$, $i=1,2$, depending only on dimension, such that the
following holds:
Letting $Z$ denote the section $\{ \tilde u \le 0\}$, assume
$|\p^\tc \tu| \in [\l, 1/\l]$ in $Z$ and $\tu = 0$ on $\p Z$. Let
$\Pi^+ \ne \Pi^-$ be parallel hyperplanes contained in $T_\ty^*\V
\setminus Z$ and supporting $Z$ from two opposite sides. If
$\diam(Z) \leq \e_c'(n)$ then
\begin{equation}\label{inf|Z|}
C_1(n) 
\frac\lambda{\gamma^-_\tc}
\leq \frac{|\inf_{Z} \tu|^n}{ \Leb{n}(Z)^2}
\leq \ C_2(n)
\frac{\gamma_\tc^+}{\lambda} 
\end{equation}
and
\begin{equation}\label{vardist}
\frac{|\tu(q)|^n}{\Leb{n}(Z)^2}
\le
C(n) \frac{\gamma_\tc^+}{\lambda}
\frac{\min\{\dist(q,\Pi^+),\dist(q,\Pi^-)\}}{ \ell_{\Pi^+}} 
\qquad \forall\,q \in \intr Z,
\end{equation}
where $\ell_{\Pi^+}$ denotes the maximal length among
all the segments obtained by intersecting $Z$ with a line
orthogonal to $\Pi^+$,
and $\gamma^\pm_\tc = \gamma^+_\tc(Z \times \V)$
is defined as in (\ref{Jacobian bound}).
%
%
\end{proposition}

\begin{proof}
Equation \eqref{vardist} follows from  Lemma~\ref{L:prop-infleq} and the assumption
$|\p^\tc \tu| \le 1/\l$.
Now, by Lemma~\ref{L:John}, we deduce that there exists an ellipsoid $E$ such that
$E \subset Z \subset nE$, where $nE$ denotes the dilation of $E$ by a factor $n$
with respect to its barycenter $\bar q$.
Taking  $\Pi^+$ and $\Pi^-$ orthogonal to one of the longest axes of $E$ and
$q=\bar q$ in \eqref{vardist} yields
\begin{align*}
|\tu(\bar q)|^n
& \leq C(n) \frac{\gamma_\tc^+}{\lambda} \frac{n}{2} \Leb{n}(Z)^2.
\end{align*}
On the other hand,  convexity of $\tu$ along the segment which crosses $Z$ and
passes through both $\bar q$ and the point $\tq$
where $\inf_Z \tu$ is attained implies
$$
|\inf_Z\tu|  \le n |\tu(\bar q)|,\\
$$
since the barycenter of $E$ divides the segment into a ratio at most $n:1$.
Combining these two estimates with \eqref{|Z| less inf varphi}
 we obtain \eqref{inf|Z|},
to complete the proof.
\end{proof}

\begin{remark}[Stability of bounds under affine
renormalization]{\rm
Noting $\gamma^\pm_\tc \le \gamma^+_c \gamma^-_c$ from Corollary
\ref{C:Jacobian transform}, we observe that the estimate
\eqref{inf|Z|} is stable under affine renormalization: let $L$ be
an affine transformation, and recall the renormalization
$$
\tu^*(q):=|\det L|^{-2/n}\tu(Lq).
$$
Then $L^{-1}(Z)$ is a section for $\tu^*$ and
$$
|\inf_{L^{-1}(Z)}\tu^*|^n=|\det L|^{-2}|\inf_Z \tu|^n \sim |\det
L|^{-2} \Leb{n}(Z)^2= \Leb{n}(L^{-1}(Z))^2.
$$
On the other hand, estimate \eqref{vardist} is not stable under
affine renormalization (a line orthogonal to $\Pi^+$ is not an
affinely invariant concept).
For this reason, both in the proof of the $c$-strict convexity
(Section~\ref{S:contact}) and in the proof of differentiability $u
\in C^1$ (Section~\ref{S:ContinuousInjection}) we apply our
Alexandrov estimates directly to the original sections, without
renormalizing them. Using this strategy, our estimates turn out to
be strong enough to adapt to our situation the strict convexity
and interior continuity theory of Caffarelli \cite{Caffarelli90}
\cite{Caffarelli92}. We perform this in the remainder of the
manuscript.}
\end{remark}
\section{The contact set is either a single point or crosses the domain}
\label{S:contact}

In this section and the final one, we complete the crucial step of
proving the strict $c$-convexity of the $c$-convex potentials
$u:\cl \U \longmapsto \R$ arising in optimal transport, meaning
$\p^c u(x)$ should be disjoint from $\p^c u(\tx)$ whenever $x,\tx
\in \U^\l$ are distinct. This is accomplished in
Theorem~\ref{T:strict convex}.
In this section, we show that, if the contact set 
does not consist of a single point, then it  extends to the
boundary of $\U$.
Our method relies on the non-negative cross-curvature \Bthree\ of
the cost $c$.

From now on we adopt the following notation: $a \sim b$ means that
there exist two positive constants $C_1$ and $C_2$, depending  on
$n$ and $\gamma^+_c \gamma^-_c/\l$ only, such that $C_1a\leq b\leq
C_2a$. Analogously we will say that $a \lesssim b$ (resp. $a
\gtrsim b$) if there exists a positive constant $C$, depending  on
$n$ and $\gamma^+_c \gamma^-_c/\l$ only, such that $a\leq C b$
(resp. $Ca \geq b$).

Recall that a point $x$ of a convex set $S \subset \R^n$ is {\em
exposed} if there is a hyperplane supporting $S$ exclusively at
$x$. Although the {\em contact set} $S := \p^\cs u^\cs(\ty)$ may
not be convex, it appears convex from $\ty$
 by Corollary \ref{C:local-global}, meaning
its image $q(S) \subset \U_\ty$ in the coordinates \eqref{q of x}
is convex.
The following theorem shows this convex set is either a singleton,
or contains a segment which stretches across the domain. We prove
it by showing the solution geometry near certain exposed points of
$q(S)$ inside $\U_\ty$ would be inconsistent with the bounds
established in the previous section.

\begin{theorem}[The contact set is either a single point or crosses the domain]
\label{thm:noexposed}
Assume \Bzero--\Bthree\ and
let $\uu$ be a $c$-convex solution of \eqref{eq:cMA} with $\U^\l \subset \U$ open.
Fix $\tx \in \U^\l$ and 
$\ty \in \p^c\uu(\tx)$, and define the contact set
$S:=\{x \in \cl \U \mid \uu(x)=\uu(\tx)-c(x,\ty)+c(\tx,\ty)\}$. 
Assume that $S \neq \{\tx \}$, i.e. it is not a singleton. Then
$S$ intersects $\p\U$.
\end{theorem}

\begin{proof}
As in Definition \ref{D:cost exponential}, we transform $(x, \uu)
\longmapsto (q, \tu)$ with respect to $\ty$, i.e. we consider the
transformation $q \in \cl \U_\ty \longmapsto x(q) \in \cl\U$,
defined on $\cl\U_\ty := -D_yc(\cl\U,\ty) \subset T^*_\ty \V$ by
the relation
$$
-D_y c(x(q),\ty)=q,
$$
and the modified cost function
$
\tilde c(q,y):=c(x(q),y)-c(x(q),\ty)
$
on $\cl \U_\ty \times \cl \V$, for which
the $\tc$-convex potential function
$
q \in \cl \U_\ty \longmapsto \tu(q):=\uu(x(q)) + c(x(q),\ty)
$
is convex.
We observe that $\tc(q, \ty) \equiv 0$ for all $q$, and moreover the set $S = \p^\cs u^\cs(\ty)$
appears convex from $\ty$, meaning $S_\ty := -D_y c(S, \ty) \subset \cl\U_{\ty}$ is convex,
by the Corollary \ref{C:local-global} to Loeper's maximum principle.

Our proof is reminiscent of Caffarelli's for the cost $\tc(q,y) =
-\<q, y\>$ \cite[Lemma 3]{Caffarelli92}. Observe $\tq := - D_y
c(\tx,\ty)$ lies in the interior of the set $\U^\l_\ty := -D_y
c(\U^\l,\ty)$ where $|\p^\tc \tu| \in
[\l/\gamma^+_c,\gamma^-_c/\l]$, according to Corollary
\ref{C:Jacobian transform}. Choose the point $\qo \in S_\ty
\subset \cl\U_\ty$ furthest from $\tq$; it is an exposed point of
$S_\ty$. We claim either $\qo = \tq$ or $\qo \in \partial U_\ty$.
To derive a contradiction,  suppose the preceding claim fails,
meaning $\qo \in U_\ty \setminus \{\tq\}$.

For a suitable choice of Cartesian coordinates on $\V$ we may,
without loss of generality,  take $\qo-\tq$ parallel to the positive $y^1$
axis. Denote by $\hat e_i$
the associated orthogonal basis for $T_\ty \V$, and set $\bo :=
\langle \qo, \hat e_1 \rangle$ and $\tb := \langle \tq, \hat e_1
\rangle$, so the halfspace $q_1 = \langle q,\hat e_1 \rangle \ge
\bo$ of $T^*_\ty \V \simeq \R^n$ intersects $S_\ty$ only at $\qo$.
Use the fact that $q^0$ is an exposed point of $S_\ty$ to cut a
corner $K_0$ off the contact set $S$ by choosing $\bs>0$ small
enough that $\bar b = (1-\bs) b^0 + \bs \tb$ satisfies:
\begin{figure}
\begin{center}
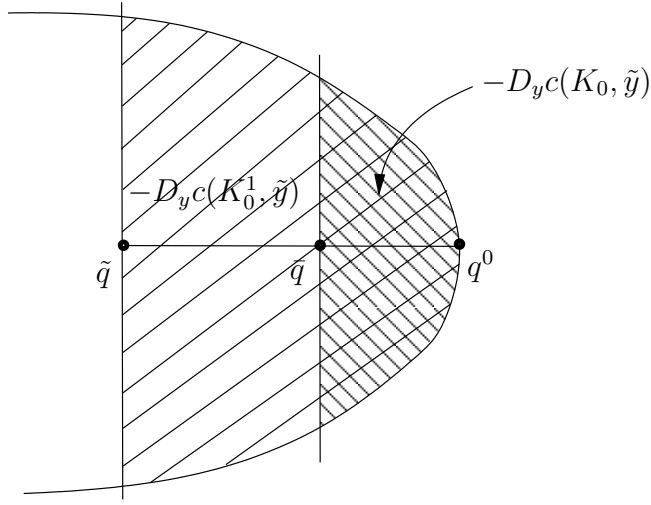\caption{{\small If the contact set $S_\ty$ has an exposed point $q_0$,
we can cut two portions of $S_\ty$
with two hyperplanes orthogonal to $\tq - q^0$. The diameter of
$-D_yc(K_0,\ty)$ needs to be sufficiently small to apply the
Alexandrov estimate Lemma \ref{L:prop-infleq}, while
$-D_yc(K_0^1,\ty)$ has to intersect $U_\ty^\l$ is a set of
positive measure to make use of Lemma \ref{lemma-infgeq} in
the case $q^0$ is not an interior point of $\supp |\p^\tc\uu
|$.}}\label{figcontactset}
\end{center}
\end{figure}
\begin{enumerate}
\item[(i)] $-D_y c(K_0,\ty):=S_{\ty}
\cap \{q \in \cl \U_\ty \mid q_1 \geq \bar b\}$ is a compact convex set in the interior of $\U_\ty$;
\item[(ii)] $\diam(-D_y c(K_0, \ty)) \le \e_c'/2$, where $\e_c'$ is from Lemma~\ref{L:prop-infleq}.
\end{enumerate}
Defining $q^s := (1-s) q^0 + s \tq$, $x^s := x(q^s)$ the
corresponding $c$-segment with respect to $\ty$, and $\bq =
q^\bs$, note that $S_\ty \cap \{q_1 = \bar b\}$ contains $\bq$,
and $K_0$ contains $\bx:= x^\bs$ and $x^0$. Since the corner $K_0$
needs not intersect the support of $|\p^c u|$ (especially,
when $q^0$ is not an interior point of $\supp |\p^c u|$),  we
shall need to cut a larger corner $K_0^1$ as well, defined by
$-D_y c(K^1_0,\ty):=S_{\ty} \cap \{q \in \cl \U_\ty \mid q_1 \geq
\tb\}$, which intersects $\U^\l$ at $\tx$.  By tilting the
supporting function slightly, we shall now define sections $K_\e
\subset K_\e^1$ of $u$ whose interiors include the extreme point
$x^0$ and whose boundaries pass through $\bx$ and $\tx$
respectively, but which converge to $K_0$ and $K_0^1$ respectively
as $\e \to 0$.

Indeed, set $y_\e := \ty + \e \hat e_1$ and observe
\begin{align}
m^s_\e (x) & := -c(x, y_\e) + c(x, \ty) + c(x^s, y_\e) - c(x^s, \ty)
\cr&= \e \langle -D_y c(x,\ty)+D_y c(x^s,\ty),\hat e_1 \rangle+o(\e)
\cr&= \e (\langle-D_y c(x,\ty), \hat e_1\rangle -(1-s)\bo - s\tb)+o(\e)
\label{e asymptotic}
\end{align}
Taking $s \in \{\bar s,1\}$ in this formula and $\e>0$ shows the sections defined by
\begin{align*}
K_\e &:=\{x \mid \uu(x) \leq \uu(\bx)-c(x,y_\e) +c(\bx,y_\e)\},\\
K^1_\e &:= \{x \mid \uu(x) \leq \uu(\tx)-c(x,y_\e) +c(\tx,y_\e)\},
\end{align*}
both include a neighbourhood of $\xo$ but
converge to $K_0$ and $K_0^1$ respectively as $\e \to 0$.
\begin{figure*}
\begin{center}
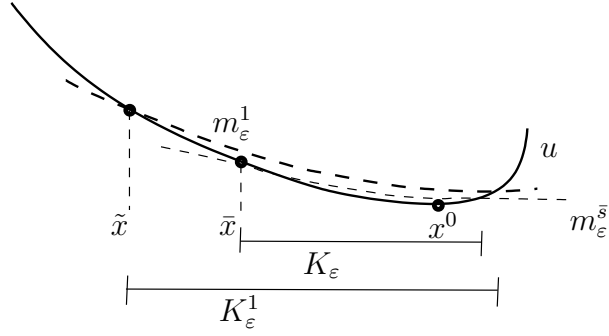\caption{{\small  We cut the graph of $u$ with the two functions $m_\e^{\bar s}$ and $m_\e^1$
to obtain two sets $K_\e \approx K_0$ and $K_\e^1 \approx K_0^1$
inside which we can apply our Alexandrov estimates to get a
contradiction (both Lemma \ref{lemma-infgeq} and Lemma \ref{L:prop-infleq}
to $K_\e$, but only Lemma \ref{lemma-infgeq} to $K_\e^1$ ). The idea is that the value of $u - m_\e^{\bar s}$
at $\xo$ is comparable to its minimum inside $K_\e$, but this is
forbidden by our Alexandrov estimates since $x_0$ is too close to
the boundary of $K_0^\e$. However, to make the argument work we
need also to take advantage of the section $K_\e^1$, in order to
``capture'' some positive mass of the $c$-Monge-Amp\`ere
measure.}}\label{figcutsection}
\end{center}
\end{figure*}

We remark that there exist a
priori no coordinates in which all set $K_\e$ are convex. However
for each fixed $\e>0$,
we can change coordinates so that both $K_\e$ and $K^1_\e$ become convex:
use $y_\e$ to make the transformations
\begin{align*}
q &:=-D_y c(x_\e(q),y_\e),
\\\tc_\e(q,y) &:=c(x_\e(q),y)-c(x_\e(q),y_\e),
\end{align*}
so that the functions
\begin{align*}
\tu_\e(q)&:=\uu(x_\e(q))+ c(x_\e(q),y_\e)-\uu(\bx)-c(\bx,y_\e), \\
\tu^1_\e(q)&:=\uu(x_\e(q))+ c(x_\e(q),y_\e)-\uu(\tx)-c(\tx,y_\e).
\end{align*}
are convex on $\U_{y_\e} := D_y c(\U,y_\e)$.
Observe that, in these coordinates, $K_\e$ and $K^1_\e$ become convex:
\begin{align*}
\tilde K_\e &:=-D_y c(K_\e,y_\e)=\{q \in \cl\U_{y_\e} \mid \tu_\e(q) \leq 0\}, \\
\tilde K^1_\e &:=-D_y c(K^1_\e,y_\e)=\{q \in \cl\U_{y_\e} \mid \tu^1_\e(q) \leq 0\},
\end{align*}
and either $\tilde K_\e \subset \tilde K^1_\e$ or $\tilde K^1_\e \subset \tilde K_\e$
since $\tu_\e(q) - \tu^1_\e(q) = const$.
For $\e>0$ small, the inclusion must be the first of the two since the limits satisfy
 $\tilde K_0 \subset \tilde K^1_0$ and $\tq \in \tilde K^1_0 \setminus \tilde K_0$.

In the new coordinates,  our original point $\tx \in \U^\l$, the
exposed point $x^0$, and the $c$-convex combination $\bx$ with
respect to $\ty$, correspond to
\begin{align*}
\tq_\e :=-D_y c(\tx,y_\e), \ \ 
q_\e^0 :=-D_y c(x^0,y_\e), \ \ 
\bq_\e :=-D_y c(\bx,y_\e). 
\end{align*}
Thanks to (ii), we have $\diam(\tilde K_\e) \leq \e_c'$  for $\e$ sufficiently small,
so that the estimate of Lemma \ref{L:prop-infleq} applies.
In these coordinates
(chosen for each $\e$) we consider the
parallel hyperplanes $\Pi^+_\e \ne \Pi^-_\e$ which support $\tilde K_\e \subset \cl\U_{y_\e}$
from opposite sides and which
are orthogonal to the segment joining $\qo_\e$ with $\bq_\e$. 
Since $\lim_{\e \to 0} \qo_\e - \bq_\e= \qo -\bq$ paralles the $\hat e_1$ axis
the limiting
hyperplanes $\Pi^\pm_0 = \lim_{\e \to 0}\Pi^\pm_\e$
must coincide with $\Pi^+_0 = \{ q \in T_\ty \V \mid q_1 = b^0\}$ and
$\Pi^-_0 = \{q \in T_\ty \V \mid q_1 = \bar b\}$.
Thus $q^0 \in \Pi^+_0$ and
$$
\frac{\dist(\qo_\e,\Pi^+_\e)}{|\qo_\e -\bq_\e|} \to 0
\ \text{ as $\e \to 0$}.
 $$
Observing that $|\qo_\e -\bq_\e|$ is shorter than segment obtained
intersecting $\tilde K_\e$ with the line orthogonal to $\Pi^+_\e$
and passing through $\qo_\e \in \intr \tilde K_\e$, Lemma
\ref{L:prop-infleq} combines with $K_\e \subset K^1_\e$ and
$|\p^{\tc_\e} \tu_\e|(K_\e) \le  \Lambda\gamma^-_c \Leb{n}(K_\e)$
from \eqref{eq:cMA} and Corollary \ref{C:Jacobian transform} to
yield
\begin{equation}\label{contradictory upper bound}
\frac{|\tu_\e(\qo_\e)|^n} { \Lambda\gamma^-_c \Leb{n}(\tilde K^1_\e)^2} \to 0 \qquad \text{as }\e \to 0.
\end{equation}

On the other hand,  $\bx \in S$ implies $\tu_\e(\qo_\e) = - m^\bs_\e(x^0)$,
and $\tx \in S$ implies $\tu^1_\e(\qo_\e) = -m^1_\e(x^0)$ similarly.
Thus \eqref{e asymptotic} yields
\begin{align}\label{comparison at q0}
\frac{\tu_\e (\qo_\e)}{\tu^1_\e (\qo_\e)}
= \frac{\e(b^0-\bar b) + o(\e)  }{\e (b^0 - \tb) + o(\e)}
\to \bs \quad {\rm as} \quad \e \to 0.
\end{align}
Our contradiction with \eqref{contradictory upper bound}--\eqref{comparison at q0}
will be established by bounding the
ratio $|\tu^1(\qo_\e)|^n/\Leb{n}(K^1_\e)^2$
away from zero.

Recall that
$$
b^0=\langle -D_y c(x^0,\ty), \hat e_1 \rangle \ =\ \max \{q_1 \mid q \in -D_y c(K_0,\ty) \}) \ >\  \tb
$$
and $\uu(x)-\uu(\tx)\geq -c(x,\ty)+c(\tx,\ty)$ with equality at $x^0$.
From the convergence of $K^1_\e$ to $K^1_0$ and the asymptotic behaviour \eqref{e asymptotic}
of $m^1_\e(x)$ we get
\begin{equation}\label{ratio var and inf var}
\begin{split}
\frac{\tu^1_\e(\qo_\e)}{\inf_{\tilde K^1_\e} \tu^1_\e}
&=\frac{-\uu(x^0) - c(x^0, y_\e)+\uu(\tx)+c(\tx,y_\e)}
{\sup_{q \in \tilde K^1_\e} [-\uu(x(q))- c(x(q),y_\e)+\uu(\tx)+c(\tx,y_\e)]}\\
& \ge \frac{- c(x^0, y_\e) + c(\tx, y_\e) + c(x^0, \ty) - c(\tx, \ty)}
{\sup_{x \in K^1_\e}[-c(x, y_\e) + c(\tx, y_\e) + c(x, \ty)- c(\tx, \ty)]} \\
&\geq \frac{\e (\langle-D_y c(x^0,\ty), e_1\rangle - \tb)+o(\e)}
 {\e (\max \{q_1 \mid q \in -D_y c(K^1_\e,\ty)\} - \tb)+o(\e) }\\
&\geq \frac{1}{2}
\end{split}
\end{equation}
for $\e$ sufficiently small.
This shows $\tu^1(q^0_\e)$ is close to the minimum value of $\tu^1_\e$. 
We would like to appeal to Lemma~\ref{lemma-infgeq}
to conclude the proof,
but are unable to do so since we only have bounds $|\p^c u| \in [\l,1/\l]$
on the potentially small intersection of $U^\l$ with $K^1_\e$.
However, this intersection occupies a stable fraction of $K^1_\e$ as
$\e \to 0$,
which we shall prove as in \cite[Lemma 3]{Caffarelli92}.

Since $K^1_\e$ converges to $K^1_0$ for sufficiently small $\e$,
observe that $K^1_\e$ is uniformly bounded.  Therefore the affine
transformation $(L^1_\e)^{-1}$ that sends $\tilde K^1_\e$ to $B_1
\subset \tilde K^{1, *}_\e \subset \cl B_n$ as in Lemma
\ref{L:John} is an expansion, i.e. $|(L^1_\e)^{-1} q -
(L^1_\e)^{-1}q'| \geq C_0|q-q'|$, with a constant $C_0>0$
independent of $\e$. Since $\tx$ is an interior point of $\U^\l$,
$B_{2\beta^-_c\d/C_0}(\tx) \subset \U^\l$ for sufficiently small
$\delta >0$, hence $B_{2\d/C_0}(\tq_\e) \subset \U^\l_{y_\e}$ with
$\beta^-_c$ from \eqref{bi-Lipschitzbound}. Defining
$\U^\l_{y_\e}:=-D_yc(\U^\l,y_\e)$, we have
$$
\U^{\l,*}_{y_\e}:=(L_\e^1)^{-1}(\U^\l_{y_\e}) \supset B_{2\d}(\tq_\e^*).
$$
Reducing $\delta$ if necessary to ensure $\delta <1$,  define (to apply Lemma~\ref{lemma-infgeq} later)
$$
\tilde K^{1,*}_{\e,\d}:=(1-\d)\tilde K^{1,*}_{\e}.
$$
(As in Lemma~\ref{lemma-infgeq}, $(1-\d)\tilde K^{1,*}_{\e}$ denotes the dilation of
$\tilde K^{1,*}_{\e}$ of a factor $(1-\delta)$ with respect to the origin.)
Since $\tilde K^{1,*}_{\e}$ is convex, it contains
the convex hull of $B_{1} \cup \{\tq^*_\e \}$, and so
$$
\Leb{n}(B_{2\d}(\tq_\e^*) \cap \tilde K^{1,*}_{\e,\d}) \ge C\d^n.
$$
for some constant $C=C(n)>0$ depending on dimension only.
Letting $\tilde K^1_{\e,\d}:=L^1_\e(\tilde K^{1,*}_{\e,\d})$
this implies
$$
\Leb{n}(\U^\l_{y_\e} \cap \tilde K^{1}_{\e,\d}) \ge C |\det L^1_\e|\d^n \sim \Leb{n}(\tilde K^1_\e)\delta^n.
$$
Recalling that $\gtrsim$ and $\lesssim$ denote inequalities which hold up
to multiplicative constants depending on $n,\lambda$ and $\gamma^+_c\gamma^+_c/\l$,
Proposition \ref{P:c-ma less ma} combines with this estimate to yield
$$
|\p \tu^1_\e|(\tilde K^1_{\e, \d})
\gtrsim|\p^{\tc_\e} \tu^1_\e|(\tilde K^1_{\e,\d}) 
\gtrsim \Leb{n}(K^1_\e) \delta^n,
$$
where (\ref{eq:cMA}) and Corollary \ref{C:Jacobian transform} have been used.
Finally, since the conclusion 
of Lemma \ref{lemma-infgeq} holds with or without stars
in light of \eqref{renormalized solution}--\eqref{compare p var with p var *},
taking $t=(1-\delta)$ in \eqref{inf geq}
yields
$$
\frac{|\inf_{\tilde K^1_\e} \tu^1_\e|^n}{\Leb{n}(\tilde K^1_\e)^2} \gtrsim \delta^{2n}.
$$
Since $\delta>0$ is independent of $\e$ this contradicts
\eqref{contradictory upper bound}--\eqref{comparison at q0} to complete the proof.
\end{proof}

\begin{remark}
{\rm 
As can be easily seen from the proof, one can actually show that if $U^\l=U$
and $S$ is not a singleton, then $S_\ty$ has no exposed points in the interior of $U_\ty$.
Indeed, if by contradiction there exists $\qo$ an exposed point of $S_\ty$
belonging to the interior of $U_\ty$, we can choose a point $\tq\in S_\ty$ in the interior of $U_\ty=U_\ty^\l$
such that the segment $\qo-\tq$ is orthogonal to a hyperplane supporting $S_\ty$ at $\qo$.
Then it can immediately checked that the above proof
(which could even be simplified in this particular case) shows that such a point $\qo$ cannot exist.
} 
\end{remark}

\section{Continuity and injectivity of optimal maps}
\label{S:ContinuousInjection}

The first theorem below combines results of Sections
\ref{S:mapping} and \ref{S:contact} to deduce strict $c$-convexity
of the $c$-potential for an optimal map, if its target is strongly
$c$-convex.  This strict $c$-convexity --- which is equivalent to
injectivity of the map
--- will then be combined with an adaptation of Caffarelli's argument \cite[Corollary 1]{Caffarelli90}
to obtain interior continuity of the map --- or equivalently
$C^{1}$-regularity of its $c$-potential function --- for
non-negatively cross-curved costs, yielding the concluding theorem
of the paper.

\begin{theorem}[Injectivity of optimal maps to a strongly $c$-convex target]
\label{T:strict convex}
Let $c$ satisfy \Bzero--\Bthree\ and \Btwos.
If $\uu$ is a $c$-convex solution of \eqref{eq:cMA} on $\U^\l \subset \U$ open,
then $\uu$ is strictly $c$-convex on $\U^\l$, meaning
$\p^c \uu (x)$ and $\p^c \uu(\tx)$ are disjoint whenever $x,\tx \in \U^\l$ are distinct.
\end{theorem}

\begin{proof}
Suppose by contradiction that $\ty \in \p^c \uu(x) \cap \p^c
\uu(\tx)$ for two distinct points $x,\tx \in \U^\l$, and set
$S=\p^\cs \uu^\cs(\ty)$. According to Theorem \ref{thm:noexposed},
the set $S$ intersects the boundary of $\U$ at a point $\bx \in \p
\U \cap \p^\cs u^\cs(\ty)$. Since \eqref{eq:cMA} asserts $\l \le
|\p^c u|$ on $\U^\l$ and $|\p^c u| \le \Lambda$ on $\cl \U$,
Theorem \ref{thm:bdry-inter}(a) yields $\ty \in \V$ (since $x, \tx
\in U^\l$), and hence $\bar x \in \U$ by Theorem
\ref{thm:bdry-inter}(b). This contradicts $\bar x \in \p \U$ and
proves the theorem.
\end{proof}

\begin{theorem}[Continuity of optimal maps to strongly $c$-convex targets]
\label{T:continuity} Let $c$ satisfy \Bzero--\Bthree\ and \Btwos.
If $\uu$ is a $c$-convex solution of \eqref{eq:cMA} on $\U^\l
\subset \U$ open, then $\uu$ is continuously differentiable inside
$\U^\l$.
\end{theorem}

\begin{proof}
Recalling that $c$-convexity implies semiconvexity, all we need to
show is that the $c$-subdifferential $\p^{c} \uu(\tx)$ of $\uu$ at
every point $\tx \in \U_\l$ is a singleton.

Assume by contradiction that is not. As $\p^{c} \uu(\tx)$ is
compact, one can find a point $y_0$ in the set $\p^c\uu(\tx)$ such
that $-D_xc(\tx,y_0) \in \p\uu(\tx)$  is an exposed point of the
compact convex set $\p\uu(\tx)$. Similarly to Definition
\ref{D:cost exponential}, we transform $(x, \uu) \longmapsto (q,
\tu)$ with respect to $y_0$, i.e. we consider the transformation
$q \in \cl \U_{y_0} \longmapsto x(q) \in \cl\U$, defined on
$\cl\U_{y_0} = -D_yc(\cl\U,y_0) + D_yc(\tx,y_0)\subset T^*_{y_0}
\V$ by the relation
$$
-D_y c(x(q),y_0)+D_yc(\tx,y_0)=q,
$$
and the modified cost function $ \tilde
c(q,y):=c(x(q),y)-c(x(q),y_0) $ on $\cl \U_{y_0} \times \cl \V$,
for which the $\tc$-convex potential function $q \in \cl \U_{y_0}
\longmapsto \tu(q):=\uu(x(q)) -\uu(\tx)+ c(x(q),y_0) -c(\tx,y_0)$
is convex. We observe that $\tc(q, y_0) \equiv 0$ for all $q$, the
point $\tx$ is sent to $\zero$, $\tu\geq \tu(\zero)=0$, and $\tu$
is strictly convex thanks to Theorem \ref{T:strict convex}.
Moreover, since  $-D_xc(\tx,y_0) \in \p\uu(\tx)$ was an exposed
point of $\p\uu(\tx)$, $0=-D_q\tc(\zero,y_0)$ is an exposed point
of $\p\tu(\zero)$. Hence, we can find a vector $v \in \p\tu(\zero)
\setminus\{0\}$ such that the hyperplane orthogonal to $v$ is a
supporting hyperplane for $\p\tu(\zero)$ at $0$.
\begin{figure}
\begin{center}
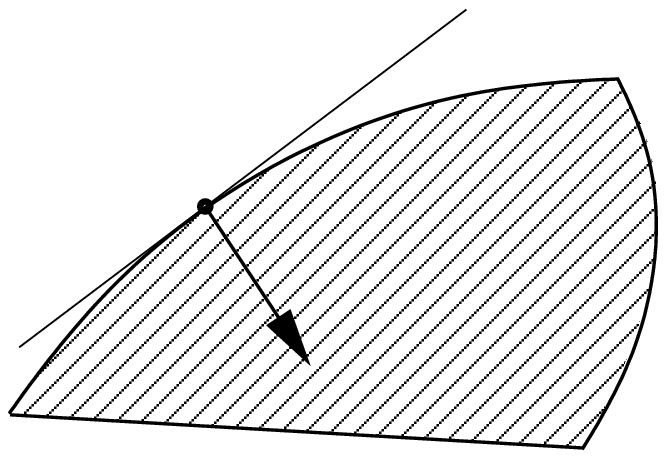\caption{{\small $v \in \p\tu(\zero)$ and the hyperplane orthogonal to $v$ is supporting $\p\tu(\zero)$ at
$0$.}}\label{figC1subdiff}
\end{center}
\end{figure}
Thanks to the convexity of $\tu$, this implies that
\begin{equation}
\label{eq:good v} \tu( -t v)=o(t) \quad \text{for }t \geq 0,\qquad
\tu(q) \geq \<v, q\>  + \tu(\zero)\quad \text{for all }q \in
\U_{y_0}.
\end{equation}
\begin{figure}
\begin{center}
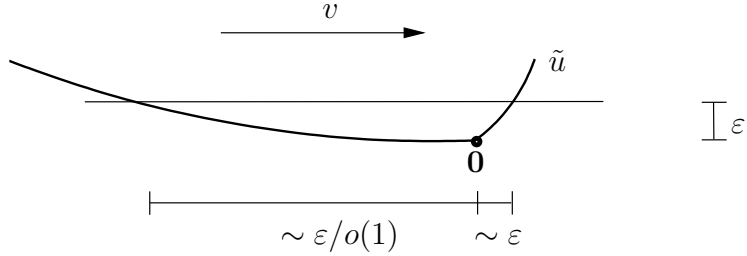\caption{{\small Since the hyperplane orthogonal to $v$ is supporting $\p\tu(\zero)$ at $0$, we have
$\tu(-tv)=o(t)$ for $t \geq 0$. Moreover, $\tu$ grows at least
linearly in the direction of $v$.}}\label{figC1sect}
\end{center}
\end{figure}
Let us now consider the section $K_\e:=\{\tu \leq \e\}$. Since
$\tu(\zero)=0$, $\tu \geq 0$ and $\tu$ is strictly convex, $K_\e
\to \{\zero\}$ as $\e \to 0$. Thus by \eqref{eq:good v} it is
easily seen that for $\e$ sufficiently small the following hold:
$$
K_\e \subset \{q \mid \<q, v\> \leq \e\},\qquad -\a(\e)v \in K_\e,
$$
where $\a(\e)>0$ is a positive constant depending on $\e$ and such
that $\a(\e)/\e \to +\infty$ as $\e \to 0$. Since $\zero$ is the
minimum point of $\tu$, this immediately implies that one between
our Alexandrov estimates \eqref{inf|Z|} or \eqref{vardist} must be
violated by $\tu$ inside $K_\e$ for $\e$ sufficiently small, which
is the desired contradiction.
\end{proof}

\end{document}